\documentclass{article}
\usepackage[margin = 1in]{geometry}
\usepackage[utf8]{inputenc}

\usepackage{graphicx}
\usepackage{amssymb,amsmath,mathtools,amscd}
\usepackage{amsthm}
\usepackage{mathabx}
\usepackage{epstopdf}
\usepackage{tikz}
\usepackage{indentfirst}
\usepackage{hyperref}
\hypersetup{
	colorlinks=true,
	linkcolor=black,
	filecolor=black,      
	urlcolor=black,
	citecolor=blue,
}

\newcommand{\CC}{\mathbb{C}}
\newcommand{\RR}{\mathbb{R}}

\newcommand{\ZZ}{\mathbb{Z}}
\newcommand{\NN}{\mathbb{N}}
\newcommand{\KK}{\mathbf{k}}
\newcommand{\LL}{\mathbf{K}}
\newcommand{\FF}{\mathbb{F}}
\newcommand{\bsep}{\beta_{\mathrm{sep}}}
\newcommand{\bfield}{\beta_{\mathrm{field}}}

\newcommand{\gfield}{\gamma_{\mathrm{field}}}
\newcommand{\Char}{\operatorname{char}}
\newcommand{\Frac}{\operatorname{Frac}}
\newcommand{\Aut}{\operatorname{Aut}}
\newcommand{\bfx}{\mathbf{x}}
\newcommand{\bfa}{\mathbf{a}}
\newcommand{\bfb}{\mathbf{b}}
\newcommand{\bfc}{\mathbf{c}}
\newcommand{\bfe}{\mathbf{e}}
\newcommand{\LM}{\mathcal{LM}}
\newcommand{\Hom}{\operatorname{Hom}}
\newcommand{\Supp}{\operatorname{Supp^\prime}}
\newcommand{\Vol}{\operatorname{Vol}}

\newtheorem{theorem}{Theorem}[section]
\newtheorem{lemma}[theorem]{Lemma}
\newtheorem{corollary}[theorem]{Corollary}
\newtheorem{proposition}[theorem]{Proposition}
\newtheorem{conjecture}[theorem]{Conjecture}
\newtheorem{question}[theorem]{Question}
\newtheorem{observation}[theorem]{Observation}

\newtheorem*{theorem*}{Theorem}
\newtheorem*{proposition*}{Proposition}
\newtheorem*{conjecture*}{Conjecture}

\theoremstyle{definition}
\newtheorem{definition}[theorem]{Definition}
\newtheorem{example}[theorem]{Example}
\newtheorem*{remark}{Remark}

\makeatletter
\newcommand{\subjclass}[2][1991]{
  \let\@oldtitle\@title
  \gdef\@title{\@oldtitle\footnotetext{#1 \emph{Mathematics Subject Classification.} #2}}
}
\newcommand{\keywords}[1]{%
  \let\@@oldtitle\@title%
  \gdef\@title{\@@oldtitle\footnotetext{\emph{Key words and phrases.} #1.}}%
}
\makeatother

\title{Degree bounds for fields of rational invariants of $\ZZ/p\ZZ$ and other finite groups}
\author{Ben Blum-Smith, Thays Garcia, Rawin Hidalgo, \\ and Consuelo Rodriguez}
\date{\today}


\begin{document}

\maketitle

\begin{abstract}
Degree bounds for algebra generators of invariant rings are a topic of longstanding interest in invariant theory. We study the analogous question for field generators for the field of rational invariants of a representation of a finite group, focusing on abelian groups and especially the case of $\ZZ/p\ZZ$. The inquiry is motivated by an application to signal processing. We give new lower and upper bounds depending on the number of distinct nontrivial characters in the representation. We obtain additional detailed information in the case of two distinct nontrivial characters. We conjecture a sharper upper bound in the $\ZZ/p\ZZ$ case, and pose questions for further investigation.
\end{abstract}

\tableofcontents

\section{Introduction}

Let $G$ be a finite group, let $\KK$ be a field of characteristic prime to $|G|$, and let $V$ be a finite-dimensional representation of $G$ over $\KK$. In this article we study the number
\[
\bfield(G,V) := \min(d:\KK(V)^G\text{ is generated  by polynomials of degree}\leq d),
\]
the minimum degree of polynomial invariants needed to generate the field of invariant rational functions $\KK(V)^G$ as a field extension of $\KK$. We focus on abelian groups, and pay special attention to the case $G=\ZZ/p\ZZ$, the cyclic group of prime order $p$. In this introduction, we explain the context and motivation for this inquiry, and present our main results.

\subsection{Context on degree bounds; goals and motivation}

\begin{paragraph}{Noether numbers.}
There is a long line of research in the invariant theory of finite groups seeking to understand the degrees of polynomials needed to generate a ring of invariants. The foundational result is Noether's \cite{noether}: if $G$ is a finite group and $V$ is a representation of $G$ over a field $\KK$ of characteristic zero, then the invariant ring $\KK[V]^G$ is generated as a $\KK$-algebra by polynomials of degree at most $|G|$. In honor of this result, the number
\[
\beta(G,V) := \min(d:\KK[V]^G\text{ is generated by polynomials of degree}\leq d),
\]
also sometimes written $\beta(\KK[V]^G)$, is known as the {\em Noether number} of the representation $V$, and Noether's original result that $\beta(G,V)\leq |G|$ is known as the {\em Noether bound}. 

There is an extensive literature studying Noether numbers. For example, Noether's restriction on the field characteristic has been partially lifted \cite{fleischmann2000noether, fogarty2001noether}: the Noether bound holds as long as $\operatorname{char} \KK \nmid |G|$ (the {\em nonmodular case}). In the {\em modular case} $\operatorname{char} \KK \mid |G|$, there is no global (i.e., independent of $V$) bound on $\beta(G,V)$, so one direction of inquiry has been the study of $\beta(G,V)$ as a function of the modular representation $V$, e.g., \cite{fleischmann2006noethermodular, symonds2011castelnuovo}. Another direction has been sharpening the Noether bound in the characteristic zero, and more generally, the nonmodular case; see for example \cite{schmid1991finite, domokos-hegedus, sezer2002sharpening, cziszter-domokos}. Yet another is the investigation of whether the Noether bound holds in noncommutative settings, which has yielded both positive \cite[Chapter~VI]{gandini2019ideals} and negative \cite{ferraro2021noether} results.
\end{paragraph}

\begin{paragraph}{Separating sets.}
Viewed as functions on the underlying vector space $V$ carrying the $G$-action, the polynomial invariants are constant along orbits, thus they can be seen as functions on the orbit space $V/G$. Because $G$ is finite, any set of generators for the invariant ring necessarily separates all orbits. Thus one motivation for studying Noether numbers is to have {\em a priori} control over the degrees of polynomials on $V$ needed to separate $G$-orbits. From the point of view of this application, Noether numbers are larger than necessary, however. In the original (2002) edition of \cite{derksen-kemper}, Derksen and Kemper introduced the notion of a {\em separating set}: a subset of an invariant ring with the same ability to separate orbits as the entire ring. Separating sets can be of lower degree than generators for the full invariant ring. For example, while the Noether bound holds for the full invariant ring only in the nonmodular case, the same numerical bound holds for separating sets in the modular case as well \cite[Corollary~3.12.3]{derksen-kemper}. Over the course of the last fifteen years, separating invariants (for both finite and infinite groups) have become the subject of significant research attention in invariant theory; for example \cite{domokos2007typical, draisma2008polarization, kemper2009separating, sezer2009constructing, dufresne2009separating, dufresne2009cohen, kohls-kraft, domokos2011helly, dufresne2013finite, kohls2013separating, dufresne2014separating, dufresne2015separating, domokos, lopatin2018minimal,  reimers2018separating, kaygorodov2018separating, derksen2020algorithms, cavalcante2020separating, reimers2020separating, lopatin2021separating, domokos2022separating, kemper2022separating}. In particular, there is now an active program \cite{kemper2009separating, kohls-kraft, dufresne2014separating, domokos, derksen2020algorithms, domokos2022separating, kemper2022separating} to study the analogue $\bsep(G,V)$ of the Noether number for separating sets, i.e., the minimum $d$ such that the invariants of degree $\leq d$ form a separating set.
\end{paragraph}

\begin{paragraph}{Main goal.}
In this article we study $\bfield(G,V)$, which is a notion for fields of rational invariants analogous to $\beta(G,V)$ for algebras of invariants. Field generation is intimately connected with orbit separation (as will be discussed momentarily), thus our object of study may also be viewed as an analogue to $\bsep(G,V)$. We view our study as a gesture toward a comprehensive program on $\bfield(G,V)$, which could be pursued in parallel with the well-established program on $\beta(G,V)$, and the younger program on $\bsep(G,V)$. We develop an approach (for the case of abelian $G$ and non-modular $V$) in Section~\ref{sec:fields-and-lattices}, present main results in Sections~\ref{sec:general} and \ref{sec:two-characters}, and, to encourage the program as a whole, we present many open questions in Section~\ref{sec:open-questions}.

To our knowledge, this is the first work that has the study of $\bfield(G,V)$ as its primary goal. Nonetheless, we are aware of some results on $\bfield(G,V)$ that have been proven in the context of other objectives; we review these below in Section~\ref{sec:prior-art}. We take the existence of these results as evidence that there may be appetite for such a program in invariant theory. 

In our view, the above discussion shows the proposed program is naturally motivated by longstanding concerns in invariant theory. That said, there is an application to signal processing, to be discussed below, which provides a much more concrete motivation. We set the stage with some general comments about the relationship of field generation to orbit separation.

Generation of the field of rational invariants (as a field extension of $\KK$) is a less restrictive condition on a set of invariants than generating the invariant ring (as a $\KK$-algebra). When $\KK$ is algebraically closed of characteristic zero, it is also less restrictive than being a separating set (as will be clear momentarily).   Nonetheless, it still provides a useful separation property. A set of invariant polynomials $f_1,\dots,f_m$ is said to {\em generically} separate orbits if there exists a $G$-stable, Zariski-open subset $U$ of $V$ on which any two orbits can be distinguished by some $f_i$. For algebraically closed $\KK$ of characteristic zero, this is equivalent by Rosenlicht's theorem to the statement that $f_1,\dots,f_m$ generate $\KK(V)^G$ as a field \cite[Lemma~2.1 and Theorem~2.3]{popov-vinberg}, i.e.,
\[
\KK(V)^G = \KK(f_1,\dots,f_m).
\]
Thus the degree of invariant polynomials needed to achieve generic separation of orbits is exactly $\bfield(G,V)$. Even if $\KK$ is not algebraically closed of characteristic zero, for example in the important case that $\KK=\RR$, the condition $\KK(V)^G = \KK(f_1,\dots,f_m)$ is still sufficient (though no longer necessary) to conclude that $f_1,\dots,f_m$ generically separate orbits on $V$, thus $\bfield(G,V)$ still bounds the needed degree. (Note that it follows from this discussion that $\bfield \leq \bsep$ when working over algebraically closed $\KK$ of characteristic zero.\footnote{This inequality fails (unsurprisingly) over finite fields $\KK=\FF_q$: the elementary symmetric polynomials on $V=\KK^n$ always form a minimal generating set for the field of rational invariants $\KK(V)^{\mathfrak{S}_n}$ of the symmetric group $\mathfrak{S}_n$, but a proper subset of these, in general not including the one of highest degree, is separating \cite{kemper2022separating, domokos2022symmetric}.}\label{note:sep-over-finite})

The above discussion of Rosenlicht's theorem remains valid if $f_1,\dots,f_m\in \KK(V)^G$ are invariant rational functions rather than polynomials. To address a question the reader may have at this point---since we are concerned with generating the whole field $\KK(V)^G$ of invariant rational functions, why do we restrict our attention only to {\em polynomial} generators in the definition of $\bfield$? 

The primary answer is that this is the definition relevant to the motivating signal processing application. The reason this application requires polynomial (rather than arbitrary rational) field generators will be discussed below. A secondary answer is that the notion of degree is natural and unambiguous for polynomials, but less so for rational functions, as $\KK(V)^G$ is not a graded object.\footnote{Our interest in generating sets for the field $\KK(V)^G$ that are contained in the ring $\KK[V]^G$ is translated in Section~\ref{sec:fields-and-lattices} into an interest in generating sets for a lattice that are contained in the positive orthant. Other recent work \cite{fukshansky2022positive} has the same interest (in lattice generators contained in the positive orthant), motivated by a completely different application.}

Nonetheless, the precedent set by the literature on degree bounds outlined above perhaps justifies interest in degree bounds on non non-polynomial field generators as well. A preliminary inquiry of this kind is undertaken in a short companion paper \cite{bbs-rational}. It uses similar methods and achieves similar results as the line of inquiry presented here.

\end{paragraph}

\begin{paragraph}{Secondary goal.}
We also study the number
\[
\gfield(G,V) = \min(d: \KK(V)^G\text{ has a transcendence basis of polynomials of degree} \leq d).
\]
By similar reasoning as above, $\gfield(G,V)$ is equal to the minimum degree of polynomials needed to identify generic orbits up to finite ambiguity. (In this case, the equality holds when $\KK=\RR$, in addition to algebraically closed fields; see \cite[Theorem~3.15]{bandeira2017estimation} for details.) Again, this is an analogue for fields to a well-studied object in invariant theory, namely
\[
\gamma(\KK[V]^G) := \min(d:\KK[V]^G\text{ has a homogeneous system of parameters of degree} \leq d),
\]
see \cite[Section~4.7]{derksen-kemper} and the references therein. 
In addition, study of $\gfield(G,V)$ is also motivated by the signal processing application to be discussed momentarily.
\end{paragraph}

\begin{paragraph}{Application to signal processing.} 
A circle of problems in signal processing involves estimating an element (signal) in a real vector space  that has been corrupted both by gaussian noise and also by transformations selected randomly from a group. Examples include {\em multi-reference alignment} \cite{perry2019sample, bandeira2020optimal, bendory2022sparse, abas2022generalized}, where one observes noisy cyclic permutations of a tuple of real numbers, and its variants \cite{abbe2017sample, bendory2022dihedral}; and {\em cryo-electron microscopy} \cite{sigworth2016principles, singer2018mathematics, bendory2020single, fan2021maximum}, were one observes noisy images of a molecule from unknown viewing directions.\footnote{In the mathematical setup for cryo-electron microscopy, the unknown viewing directions are expressed as action by random elements of $SO(3)$, followed by a fixed projection. The projection introduces some complication into the story that follows, which we elide for the sake of brevity, but see \cite{bandeira2017estimation} for a detailed discussion.}  

It is shown in \cite{bandeira2017estimation} that, for high noise levels, the number of samples needed, in an information-theoretic sense, to accurately estimate the orbit of a generic (respectively worst-case) signal in situations such as these, depends exponentially on  the degrees of the invariant polynomials needed to achieve generic (respectively complete) separation of orbits. From \cite[Theorem~2.15]{bandeira2017estimation} we learn that if an orbit is uniquely identified by the polynomials of degree $\leq d$, then it can be estimated using $O(\sigma^{2d})$ samples, where $\sigma$ is the noise level. Conversely, from \cite[Theorem~2.16]{bandeira2017estimation} we learn that if two orbits are not distinguished by the invariant polynomials of degree up to $d-1$, then they cannot be reliably distinguished on the basis of fewer than $\Omega(\sigma^{2d})$ samples. In view of the above discussion, this can be rephrased as saying that for this type of problem, $\bfield(G,V)$ bounds (from above) the complexity of recovering a generic orbit from samples, while $\bsep(G,V)$ determines (i.e., bounds from below and above) that of recovering a worst-case orbit. 

The reason these results concern invariant {\em polynomials}, rather than more general invariant rational functions, is that invariant polynomials can be estimated from samples. A crucial step in the proof of \cite[Theorem~2.15]{bandeira2017estimation} is to produce an unbiased estimator for (i.e., a function of the observed samples whose expectation is equal to) the value of a degree-$d$ invariant polynomial evaluated on an orbit \cite[Section~6.1.1]{bandeira2017estimation}. The $\sigma^{2d}$ shows up in a bound on the variance of this estimator. On the other hand, because the observed samples involve gaussian noise, any function with an unbiased estimator is expressible as a convolution with a gaussian, so it must be analytic on the entire signal domain. In particular, rational functions do not have unbiased estimators due to their poles.

As we will see, $\bfield(G,V)$ can be much lower than $\bsep(G,V)$. For example:
\begin{itemize}
    \item For $G=\ZZ/n\ZZ$ ($n$ a natural number) and $V$ its regular representation over $\CC$, it follows from \cite[Theorem~2.1]{domokos} that $\bsep(G,V) = n$. On the other hand, by \cite[Theorem~4.1]{bandeira2017estimation}, discussed a little more in Section~\ref{sec:prior-art} below, if $G$ is any finite abelian group of order at least 3 and $V$ is its regular representation, then $\bfield(G,V) = 3$, regardless of the group.
    \item For $G=\ZZ/p\ZZ$ and $V$ any nontrivial representation over $\CC$, we have $\bsep(G,V) = p$, again by \cite[Theorem~2.1]{domokos}. (We verify below in Proposition~\ref{prop:separating-over-R} that this also holds over $\RR$, the case relevant to the present application.) On the other hand, Theorem~\ref{thm:upper-bound} below shows that if $V$ contains sufficiently many distinct nontrivial characters, then $\bfield(G,V)$ cannot be much bigger than $p/2$, and computational evidence suggests (see Conjecture~\ref{conj:sharp-upper}) that actually it is smaller still.
\end{itemize}
\noindent So there is a significant advantage in this context to working with generic rather than worst-case signals. This motivates an understanding of $\bfield(G,V)$. 

The quantity $\gfield(G,V)$, defined above, which by the same reasoning determines the sample complexity of estimating the orbit of a generic signal up to a finite ambiguity, can be even lower. This motivates an understanding of $\gfield(G,V)$. 
\end{paragraph}

\subsection{Results and methods}

\begin{paragraph}{Results.}
We prove new lower bounds on $\gfield(G,V)$ for arbitrary $G$, and new upper bounds on $\bfield(G,V)$ for $G=\ZZ/p\ZZ$, with $p$ prime. (Since $\gfield \leq \bfield$ always, the lower bounds also bound $\bfield$, and the upper bounds also bound $\gfield$.) The main results are these:

\begin{theorem*}[Lower bound---Theorems~\ref{thm:lower-bound} and \ref{thm:lower-bound-general}]
For any finite group $G$ and any faithful representation $V$ of dimension $N$, we have
\[
\gfield(G,V) \geq \sqrt[N]{|G|}.
\]
If $G$ is abelian and $V$ is non-modular, then
\[
\gfield(G,V) \geq \sqrt[m]{|G|},
\]
where $m$ is the number of distinct, nontrivial characters of $G$ occurring in $V$.
\end{theorem*}

The quantity $m$ in the theorem statement will be used frequently in our main (abelian, non-modular) situation. If the ground field $\KK$ does not contain enough roots of unity to diagonalize the action of $G$ on $V$, then $m$ should be understood to refer to the cardinality of the set of distinct, nontrivial characters that appear in $V$ after base-changing it to an extension $\LL\supset \KK$ that does diagonalize the action.

\begin{theorem*}[Upper bound---Theorem~\ref{thm:upper-bound}]
For $p\geq 3$ prime, and $V$ a finite-dimensional non-modular representation of $G=\ZZ/p\ZZ$ containing at least $3$ distinct nontrivial characters, we have
\[
\bfield(G,V)\leq \frac{p+3}{2}.
\]
\end{theorem*}

The lower bound is sharp, and under mild conditions we fully characterize the extremal groups and representations (Proposition~\ref{prop:lower-bound-attained}). We also prove refinements in the abelian, non-modular situation. When the $m$ of the theorem statement is large relative to $\log |G|$, the lower bound stated above is very low, but we verify that for virtually all abelian groups $G$, if $V$ is faithful then $\gfield$ does not drop below $3$ (Corollary~\ref{cor:gfield-at-least-3}). Also, when $G=\ZZ/p\ZZ$ and $m=2$, we show the above lower bound can be improved by $1$ plus rounding error, and this is sharp, and we characterize the primes $p$ and representations $V$ that attain this slightly improved bound (Proposition~\ref{prop:Z-mod-p-lower}). 

The upper bound is of theoretical interest since it is lower than the Noether bound for  large classes of representations of $G=\ZZ/p\ZZ$ for which $\beta, \bsep$ never drop below the Noether bound.\footnote{If $G=\ZZ/p\ZZ$, {\em any} faithful non-modular representation $V$ has $\beta(G,V)=p$. The same holds for $\bsep(G,V)$ if $\KK$ is algebraically closed, by \cite[Theorem~2.1]{domokos}, and actually the argument works as long as $\KK$ contains $p$th roots of unity. To round out this story, and particularly with an eye to the signal processing application discussed above, we verify that $\bsep(G,V)=p$ as well if $\KK = \RR$; see Proposition~\ref{prop:separating-over-R} below.}  However, it is not sharp. We provide some evidence for the following:

\begin{conjecture*}[Sharp upper bound---Conjecture~\ref{conj:sharp-upper}]
If $G=\ZZ/p\ZZ$, $V$ is non-modular, and $m$ is the number of distinct, nontrivial characters of $G$ occurring in $V$, then
\[
\bfield(G,V) \leq \left\lceil \frac{p}{\lceil m/2\rceil} \right\rceil.
\]
\end{conjecture*}

\noindent We also exhibit representations that attain this conjectural bound (Proposition~\ref{prop:extremal}).

When $G=\ZZ/p\ZZ$ and $m=2$, we obtain more detailed results. We give an upper bound on $\bfield$ that becomes an exact formula under an easy-to-verify hypothesis  (Proposition~\ref{prop:q+b+r-1} and the remark following it). Using this, we show that $\bfield \leq (p+3)/2$ except in the special case that the two characters in $V$ are inverses (Proposition~\ref{prop:(p+3)/2}); this is a key lemma for the above upper bound. And in the special case that $V$ contains no trivial or repeated characters, we provide some information about the form of the Hilbert series of the invariant ring $\KK[V]^G$ (Proposition~\ref{prop:hilbert-series}).

We give a few other results that partially explain a tendency  observed in computational data in the $G=\ZZ/p\ZZ$ case. In general, $\gfield$ and $\bfield$ are not equal, but they nonetheless were equal in many examples we computed.  Propositions~\ref{prop:b=g-when-close-to-bottom} and \ref{prop:b=g-in-2d}  note conditions under which this is guaranteed. (Proposition~\ref{prop:extremal} is also an example of this, although its primary purpose is to show that the conjectural upper bound discussed above is sharp.) 
\end{paragraph}

\begin{paragraph}{Methods.}
For abelian $G$ and non-modular $V$, our main case, the basic strategy employed here is to transform the calculation of $\bfield(G,V)$ into a question about a sublattice of $\ZZ^m$, and then analyze this lattice. This follows a standard approach in the invariant theory of finite abelian groups, of  diagonalizing the action so as to be able to view the invariant ring as a normal affine semigroup ring, allowing the ring to be studied by looking at the underlying semigroup  \cite{huffman, schmid1991finite, bruns-herzog, smith1996noether, neusel-sezer, domokos, gandini2019ideals}; the ambient group of the affine semigroup, which often plays a role in these analyses, is exactly the lattice we study. In at least two prior works \cite{hubert2013scaling, hubert-labahn}, this strategy has been used to study fields of rational invariants. Our simultaneous focus on degree bounds and fields puts slightly different demands on the setup than found in the works cited, so we give a self-contained account of the reduction to the lattice problem.

The lattices themselves are studied with a variety of methods. The lower bound on $\gfield$ for abelian $G$ is proven with a geometry of numbers-typed argument. The upper bound for $G=\ZZ/p\ZZ$ and $m\geq 3$ is proven by bootstrapping from detailed information about the $m=2$ case. Both the $m=2$ results and the bootstrapping technique are based on careful analysis of equations for the lattices, as are most of the other results mentioned above.

Different methods are needed for the handful of results we obtain for not-necessarily-abelian $G$. The lower bound on $\gfield$ in this more general setting is deduced from of a lemma of Gregor Kemper \cite{kemper1996constructive} that is in turn based on a generalization of B\'ezout's theorem. The characterization of the groups and representations that attain this bound also involves the Chevalley-Shepard-Todd Theorem and the classification of complex reflection groups.
\end{paragraph}

\subsection{Context on generation of invariant fields; prior art on $\bfield$}\label{sec:prior-art} We contextualize the present inquiry within previous work on (not necessarily polynomial) generators for fields of invariant rational functions, and then discuss priorly known degree bounds on polynomial generators.

Explicit constructions of generators for invariant fields of various permutation groups are part of classical Galois theory. For example, given a set of $n$ indeterminates $x_1,\dots,x_n$, the elementary symmetric polynomials generate the rational invariants of the canonical action of the symmetric group $\mathfrak S_n$ over the coefficient field; these and the Vandermonde determinant generate the rational invariants of the alternating group $\mathfrak A_n\subset\mathfrak S_n$;\footnote{This assertion for $\mathfrak A_n$ requires the hypothesis that the ground field has characteristic different from 2; this defect can be remedied by replacing the Vandermonde determinant with the sum only of its positive terms.} and the elementaries together with a root of the resolvent cubic (of the univariate quartic with roots $x_1,\dots,x_4$) generate the rational invariants of the dihedral group $\mathfrak D_4\subset \mathfrak S_4$; all of this was in essence known to Galois.\footnote{In these classical cases, it happens that the given generators are even polynomials that generate the invariant ring as an algebra. However, in the original Galois-theoretic context, the emphasis was on their role as field generators.} Another venerable source of explicit constructions is Burnside's classic 1911 text on finite groups \cite{burnside1911theory}, which devotes Chapter XVII to the study of fields of rational invariants, computing  generators for a number of linear (but not necessarily permutation) actions such as the 3-dimensional irreducible representations of $\mathfrak A_5$ and $PSL(2,7)$ over $\CC$.

Further explicit constructions, again with Galois-theoretic motivation, have been given in various special cases in the context of another longstanding research program bearing Noether's name, the so-called {\em Noether's problem}, which asks when fields of rational invariants of finite groups are purely transcendental;\footnote{This famed problem was originally posed for permutation groups by Noether  \cite{noether1913rationale}; it has connections to Galois theory and birational geometry. See \cite{saltman1985groups, formanek1984rational} for overviews,  \cite{swan1983noether} for a Galois-theoretic point of view, and \cite{bogomolov1987brauer} for a contribution from the birational geometry side. Much of this work is concerned with giving inexplicit obstructions to a field of rational invariants being pure transcendental, rather than giving explicit generators.} some examples are  \cite{charnow1969fixed, kemper1996constructive}. 
One can also sometimes find explicit constructions of field generators in work belonging properly to invariant theory, such as \cite{thiery2000algebraic}, which gives an elegant construction of field generators in a situation where a satisfactory description of algebra generators remains out of reach.

More recently, researchers have developed general algorithms to find generators for fields of rational invariants  \cite{muller1999calculating, hubert-kogan, kemper2007computation}. For finite groups, a uniform, characteristic-free, explicit construction is given in \cite{fleischmann2007homomorphisms}. Hubert and her collaborators have also given a variety of efficient algorithms adapted to specific important groups and representations \cite{hubert2012rational, hubert-labahn, gor-hub-pap}, as well as applications to differential geometry and dynamical systems \cite{hubert2007smooth, hubert2009differential, hubert2013scaling}. 

In contrast to the present study, it is not a goal of any of this work to look for generating sets consisting of polynomials of minimal degree. Still, a handful of results on what we here call $\bfield(G,V)$ have been drawn from it.

Degree bounds are noted as a consequence of constructions of polynomial generators in both \cite{fleischmann2007homomorphisms} and \cite{hubert-labahn}. In \cite[Corollary~2.3]{fleischmann2007homomorphisms}, the explicit construction of polynomial generators is used to conclude that $\bfield(G,V)\leq |G|$.\footnote{The point is that this holds regardless of the characteristic, unlike the classical Noether bound for algebra generators.} In \cite{hubert-labahn}, the authors give an algorithm to compute polynomial (actually monomial) generators for $\KK(V)^G$, in the special case of abelian $G$ and non-modular $V$.\footnote{The algorithm is based on integer linear algebra, and along with generators it also yields an explicit rule to write an arbitrary invariant rational function in terms of those generators.} This construction leads to a bound $\bfield(G,V)\leq |G|/\det H$, where $H$ is a certain matrix that depends on the way the action of $G$ on $V$ is presented (see \cite[p.~3038]{hubert-labahn}). 

The matrix $H$ of \cite{hubert-labahn} is always the $1\times 1$ matrix $\begin{pmatrix}1\end{pmatrix}$ for a faithful representation of $\ZZ/p\ZZ$, so the bound of \cite{hubert-labahn} is equal to the bound of \cite{fleischmann2007homomorphisms} in this case, and both are equal to the Noether bound. Thus the upper bound given in the previous subsection represents an improvement on the known bounds for most representations of  $\ZZ/p\ZZ$.

A provocative theorem about $\bfield(G,V)$ is proven in \cite{bandeira2017estimation}, in the context of the application to signal processing discussed above. For $G$ finite abelian, $\KK$ of coprime characteristic, and $V=V_{\mathrm{reg}}$ the regular representation, it is shown in \cite[Theorem~4.1 and the remark following]{bandeira2017estimation} that $\bfield(G,V)\leq 3$.\footnote{In \cite{bandeira2017estimation} this is argued under the assumption that $\KK$ contains $|G|$th roots of unity, but we will see below in Lemma~\ref{lem:base-change} that this additional hypothesis is superfluous. Using the methods of the present work it can be shown that the inequality $\bfield(G,V_{\mathrm{reg}}) \leq 3$ can be sharpened to equality if $|G|\geq 3$.} This result reveals a contrast in behavior between $\bfield$ and $\beta, \bsep$. First and most strikingly, the bound $\bfield(G,V_\mathrm{reg}) \leq 3$ is independent of the abelian group $G$. Secondly, this bound reveals that $\bfield(G,V)$ has a tendency to trend downward as $V$ grows, with the regular representation almost always attaining a minimum among faithful representations.\footnote{To illustrate, when $G$ is cyclic and $V$ is one-dimensional and faithful we have $\bfield(G,V)=|G|$, as there are no invariants of degree less than $|G|$. So in this situation, as one adds the other characters of $G$ to $V$, $\bfield$ must drop from $|G|$ to $3$. By saying that $\bfield$ trends downward as $V$ grows we do not intend to make a precise statement, as $\bfield$ is not a strictly nonincreasing function of $V$; see Example~\ref{ex:b-not-monotonic} below. Some light is shed on the trend by Proposition~\ref{prop:subsets}. On the other hand, we can make precise the statement that the regular representation almost always minimizes $\bfield(G,V)$ among faithful representations of abelian $G$: this occurs unless $G$ is an elementary abelian $2$-group, by \cite[Theorem~4.1]{bandeira2017estimation} combined with Corollary~\ref{cor:gfield-at-least-3} below.} For context, the traditional Noether number $\beta(G,V)$ is nondecreasing as the representation $V$ grows, and attains its maximum value on the regular representation, at least in characteristic $0$ \cite[Corollary~6.3]{schmid1991finite} or greater than $|G|$ \cite{smith2000theorem}. Similarly, in the non-modular situation, $\bsep(G,V)$ can only grow with $V$ \cite[Proposition~2]{kohls-kraft}, \cite[Theorem~2.4.9]{derksen-kemper}, and it attains its maximum value on the regular representation, at least if $\KK$ is infinite \cite[Theorems~2.3(b) and 2.4]{draisma2008polarization}.\footnote{To spell out this last point: In the nonmodular situation, the group algebra $\KK[G]$ is semisimple, so Maschke's theorem obtains, and every representation $V$ is the sum of irreducibles. The cited \cite[Theorems~2.3(b) and 2.4]{draisma2008polarization} show that if $\KK$ is an infinite field, then $\bsep(G,V)$ is not increased by increasing the multiplicities of the irreducibles occuring in $V$, while \cite[Proposition~2]{kohls-kraft} (equivalently,  \cite[Theorem~2.4.9]{derksen-kemper}) shows that it is not decreased either. Thus in this situation, $\bsep(G,V)$ depends only on the set of distinct irreducibles in $V$ and not on their multiplicities. By Artin-Wedderburn theory, $V_\mathrm{reg} \cong \KK[G]$ contains every irreducible representation at least once, and a second application of \cite[Proposition~2]{kohls-kraft} allows us to conclude that $V_\mathrm{reg}$ maximizes $\bsep(G,V)$.} All of this was key inspiration for the present inquiry.

The paper \cite{kemper1996constructive} of Kemper, mentioned above in the context of Noether's problem, should be highlighted for an additional reason as context for the present work. Although the field generating sets it constructs are not polynomials, the degrees of the numerators and denominators play a key role in the arguments. Furthermore, the most general lower bound on $\bfield(G,V)$ proven in the present work, Theorem~\ref{thm:lower-bound-general}, is a straightforward consequence of the fundamental lemmas proven there.\\

The structure of the paper is as follows. In Section~\ref{sec:fields-and-lattices}, which is exclusively focused on the main (abelian, non-modular) case, we prove the equivalence of the calculation of $\bfield$ and $\gfield$ with questions about lattices. We also show that $\bfield$ and $\gfield$ depend only on the number of distinct, nontrivial characters of $G$ occurring in $V$ (not their multiplicities). The section also serves to fix notation. In Section~\ref{sec:general}, we prove the paper's general results: the lower bounds for arbitrary $G$, the upper bound for $G=\ZZ/p\ZZ$, and results directly connected to these (such as the characterization of $G$ and $V$ attaining the lower bound). The upper bound is proven modulo Proposition~\ref{prop:(p+3)/2}, whose proof is deferred to Section~\ref{sec:two-characters} where it fits in better. Section~\ref{sec:two-characters} proves our results about representations of $G=\ZZ/p\ZZ$ with exactly two nontrivial isotypic components.  Section~\ref{sec:conjecture} concerns the conjectural upper bound on $\bfield$ for $G=\ZZ/p\ZZ$ mentioned above. Section~\ref{sec:other-questions} collects together many other open questions. 

\section{Invariant fields and lattices}\label{sec:fields-and-lattices}

With just a few exceptions, all our arguments are based on an equivalence between the problem of finding $\bfield(G,V)$, and a question about sublattices of the integer lattice $\ZZ^m\subset\RR^m$. As mentioned in the Methods section, this connection comes from diagonalizing the group action, a standard approach in the invariant theory of finite abelian groups. However, our simultaneous focus on degree bounds and fields puts slightly different demands on the setup than found in the works cited above.
Therefore, 
in this section, we give a self-contained account of the reduction to the lattice question; we note connections with prior work in remarks throughout. 
We then use the lattice point of view to show that $\bfield(G,V)$ depends only on the set of distinct, nontrivial characters in the representation $V$ (not their multiplicities).

This section also serves to fix notation.

\subsection{Basic setup and reduction to lattice problem}

Notation used throughout the section and/or paper is introduced in bulleted lists for ease of visual access.
\begin{itemize}
    \item $G$ is a finite group. It is almost always abelian (exceptions: Lemma~\ref{lem:base-change}, Theorem~\ref{thm:lower-bound-general}, Proposition~\ref{prop:lower-bound-attained}, and various questions in Section~\ref{sec:other-questions}).
    
    \item $\KK$ is a field. It is usually of characteristic prime to $|G|$ (exceptions: Lemma~\ref{lem:base-change}, Theorem~\ref{thm:lower-bound-general}, and various questions in Section~\ref{sec:other-questions}).
    
    \item $V$ is a finite-dimensional, faithful representation of $G$ over $\KK$; its dimension is $N$. The condition $\Char \KK \nmid |G|$ is also indicated by saying that $V$ is {\em non-modular}; $V$ is usually non-modular (with the same exceptions as the previous bullet).
    
    \item $\KK[V]$ is the ring of polynomial functions on $V$.
    
    \item $\KK(V)$ is the field of rational functions on $V$, i.e., the fraction field of $\KK[V]$.
    
    \item The action of $G$ on $\KK[V]$, respectively $\KK(V)$, is defined by $(gf)(v) = f(g^{-1}v)$ for $g\in G$, $v\in V$, and $f\in \KK[V]$, respectively $\KK(V)$.
    
    \item $\KK[V]^G := \{ f\in \KK[V]: gf = f\text{ for all }g\in G\}$ is the ring of polynomial invariants.
    
    \item $\KK(V)^G := \{ f\in \KK(V): gf = f\text{ for all }g\in G\}$ is the field of rational invariants.
    
    \item For a given natural number $d$,
\[
\KK[V]^G_{\leq d} := \{ f\in \KK[V]^G: \deg f\leq d \}
\]
is the $\KK$-vector space of polynomial invariants of degree $d$ or less. 
    \item As above, 
    \[
\bfield(G,V) := \min(d: \KK(V)^G = \KK(\KK[V]^G_{\leq d}))\]
is the minimum $d$ such that the polynomial invariants of degree $\leq d$ generate $\KK(V)^G$ over $\KK$ as a field.\footnote{The same concept is defined, with slightly different notation, in \cite[Definition~2.2]{fleischmann2007homomorphisms}. Our notation is inspired by $\bsep(G,V)$, the minimum number such that polynomials of at most that degree form a separating set; see for example \cite{kohls-kraft, domokos}.}

    \item We also consider the number
    \[
\gfield(G,V) := \min(d: [\KK(V)^G:\KK(\KK[V]^G_{\leq d})]<\infty),\]
the minimum $d$ such that $\KK(V)^G$ is a finite field extension of the field generated by the invariants of degree $\leq d$.\footnote{The notation here is inspired by the notation $\gamma(\KK[V]^G)$ for the minimum $d$ such that the invariant ring $\KK[V]^G$ is finite over the subring generated by polynomials of degree $\leq d$ \cite[Definition~4.7.1]{derksen-kemper}.  The latter number is also called $\sigma(G,V)$ \cite{cziszter2013generalized, elmer2014zero, elmer2016zero}, so $\sigma_{\mathrm{field}}(G,V)$ would have been an alternative.} 
\end{itemize}

\begin{remark}
Because $G$ is finite, $\KK(V)^G$ is the fraction field of $\KK[V]^G$, so at the very least, $\KK[V]^G$ generates $\KK(V)^G$ as a field. Therefore, $\bfield(G,V)$ is well-defined, and in fact bounded above by $\beta(G,V)$, the Noether number of $V$. It is immediate from the definitions that
\[
\gfield(G,V) \leq \bfield(G,V).
\]
Because $[\KK(V):\KK(V)^G]=|G|<\infty$, we could alternatively have defined $\gfield(G,V)$ as the minimum $d$ such that $\KK(V)$ (rather than $\KK(V)^G$) is finite over the subfield $\KK(\KK[V]^G_{\leq d})$. Because $\KK(V)$ is a finitely generated field extension of $\KK$, it is finite over $\KK(\KK[V]^G_{\leq d})$ if and only if it is algebraic over the latter. Thus a third equivalent characterization of $\gfield(G,V)$ is as the minimum $d$ such that $\KK[V]^G_{\leq d}$ contains a transcendence basis for $\KK(V)$ over $\KK$.
\end{remark}

First, we verify that no generality is lost by adjoining roots of unity to the ground field. 
\begin{itemize}
    \item For any field extension $\LL$ of $\KK$, write $V_\LL := \LL\otimes_\KK V$, the base-change of $V$ to $\LL$, equipped with the natural action of $G$ resulting from its action on the second tensor factor.
\end{itemize}

\begin{remark}
It is common in the literature on invariants of finite abelian groups to work over an algebraically closed field (e.g., \cite{schmid1991finite, domokos}), or at least a field already containing the relevant roots of unity (e.g., \cite{hubert-labahn, gandini2019ideals}). However, it is recognized (see for example Section~4 of \cite{knop2004noether} or the comments at the beginning of Section~4.3 in \cite{cziszter2016interplay}) that many results proven at this level of generality hold in greater generality. The following lemma is in the spirit of this recognition. It involves more bookkeeping than analogous results for rings because fields of rational invariants are not direct sums of their degree components. Still, it is essentially routine.
\end{remark}

\begin{lemma}\label{lem:base-change}
Let $\LL/\KK$ be any field extension. Then
\[
\bfield(G,V) = \bfield(G,V_\LL)
\]
and
\[
\gfield(G,V) = \gfield(G,V_\LL).
\]
\end{lemma}

\begin{proof}
We can extend the inclusion $\KK\hookrightarrow \LL$ into an embedding of $\KK(V)$ into the field $\LL(V_\LL)$ of rational functions on $V_\LL$ with coefficients in $\LL$. We view all of what follows as taking place inside this latter field. With this setup, $\LL(V_\LL)$ is the composite of its subfields $\KK(V)$ and $\LL$, and these subfields are linearly disjoint over $\KK$.

For any natural number $d$, $\LL[V_\LL]_{\leq d}$ is the $\LL$-span of $\KK[V]_{\leq d}$ in $\LL(V_\LL)$. The functor of invariants commutes with the flat base change from $\KK$ to $\LL$, so 
\[
\LL[V_\LL]_{\leq d}^G = \left(\LL\otimes_\KK \KK[V]_{\leq d}\right)^G \cong \LL\otimes_\KK \KK[V]_{\leq d}^G,
\]
and we conclude $\LL[V_\LL]^G_{\leq d}$ is the $\LL$-span of $\KK[V]^G_{\leq d}$ as well. (Base-changing to $\LL$ does not yield any ``extra" low-degree invariants.)

For the first equality, it will suffice to show that for any natural number $d$, the inequality $\bfield(G,V) \leq d$ implies $\bfield(G,V_\LL)\leq d$ and vice versa.

If $\bfield(G,V)\leq d$, then $\KK[V]^G_{\leq d}$ generates $\KK(V)^G$ as a field. Thus the subfield 
\[
\LL(\LL[V_\LL]^G_{\leq d})
\]
of $\LL(V_\LL)^G$ generated by the degree $\leq d$ invariants contains the entirety of $\KK(V)^G$, since $\LL[V_\LL]^G_{\leq d}$ contains $\KK[V]^G_{\leq d}$. In particular, $\LL(\LL[V_\LL]^G_{\leq d})$ contains $\KK[V]^G\subset \KK(V)^G$. Since it also contains $\LL$, it contains the $\LL$-span of $\KK[V]^G$, which is $\LL[V_\LL]^G$ as above. As $\LL(V_\LL)^G = \Frac \LL[V_\LL]^G$ (because $G$ is finite), we can conclude that $\LL(\LL[V_\LL]^G_{\leq d})$ is actually all of $\LL(V_\LL)^G$. Thus $\bfield(G,V_\LL)\leq d$.

Conversely, suppose that $\bfield(G,V_\LL)\leq d$, i.e., $\LL(V_\LL)^G$ is generated by $\LL[V_\LL]^G_{\leq d}$. Choose a basis $B$ for $\LL$ as a $\KK$-vector space, and let $f\in \KK(V)^G$ be arbitrary. Since $\KK(V)^G = \Frac \KK[V]^G$, we have
\[
f = P/Q
\]
with $P,Q\in \KK[V]^G$ and $Q$ nonzero (but we do not have control over the degrees of $P$ and $Q$). Meanwhile, since $f\in \KK(V)^G\subset \LL(V_\LL)^G$, and $\LL(V_\LL)^G$ is generated by $\LL[V_\LL]^G_{\leq d}$, we can also represent it as
\[
f = L/M
\]
with $L$ and $M$ polynomial expressions in the elements of $\LL[V_\LL]^G_{\leq d}$ with coefficients in $\LL$, and $M$ nonzero. Writing each element of $\LL[V_\LL]^G_{\leq d}$ appearing in these expressions as an $\LL$-linear combination of elements of $\KK[V]^G_{\leq d}$, we may view $L$ and $M$ as polynomials in the elements of $\KK[V]^G_{\leq d}$ with coefficients in $\LL$. Then, since each coefficient from $\LL$ can be expressed in terms of the basis $B$, we may write
\[
L = \sum_{b\in B} L_bb,\;\, M = \sum_{b\in B} M_bb,
\]
with $L_b, M_b$ polynomial expressions in the elements of $\KK[V]^G_{\leq d}$ with coefficients in $\KK$, and both sums finitely supported. Since $P/Q=f=L/M$, we have
\[
\sum_{b\in B} PM_bb = \sum_{b\in B} QL_b b.
\]
Because $P,Q,L_b,M_b\in \KK(V)$, and $\KK(V)$ is linearly disjoint from $\LL$ over $\KK$, the linear independence of the $b$'s over $\KK$ implies that
\[
PM_b = QL_b
\]
for all $b\in B$. Since $M$ is nonzero, there is at least one $b$ such that $M_b$ is nonzero; as $Q$ is also nonzero, for this particular $b$ we get
\[
f = P/Q = L_b / M_b.
\]
The right side represents $f$ as a rational function in the elements of $\KK[V]^G_{\leq d}$ with coefficients in $\KK$, so $f\in \KK(\KK[V]^G_{\leq d})$, and we conclude that $\bfield(G,V) \leq d$. This completes the proof that $\bfield(G,V) = \bfield(G,V_\LL)$.

For the second equality (involving $\gfield$), what has to be shown is that $\KK[V]^G_{\leq d}$ contains a transcendence basis of $\KK(V)/\KK$ if and only if $\LL[V_\LL]^G_{\leq d}$ contains a transcendence basis of $\LL(V_\LL)/\LL$. If $f_1,\dots,f_N \in \KK[V]^G_{\leq d}$ are $\KK$-algebraically independent, then they remain algebraically independent over $\LL$, since any polynomial relation in them over $\LL$ can be split into a finite set of polynomial relations over $\KK$ by writing each coefficient on the basis $B$, just as above. 

In the other direction, suppose $f_1,\dots,f_N \in \LL[V_\LL]^G_{\leq d}$ are $\LL$-algebraically independent. Write each one as a finite sum
\[
f_i = \sum_{b\in B} f_{i,b} b
\]
with each $f_{i,b}$ in  $\KK[V]^G_{\leq d}$. Then there must be a subset of the components $f_{i,b}$ of cardinality $N$  that are algebraically independent over $\KK$: if not, every $N$ of them would have an algebraic relation over $\KK$, which would also hold over $\LL$, thus the field $\LL(\{f_{i,b}\})$ they generate over $\LL$ would have transcendence degree $<N$; but it contains the $N$ algebraically independent $f_i$, a contradiction. This completes the proof.
\end{proof}

\begin{remark}
The proof of this lemma goes through unchanged without the assumption either that $G$ is abelian or that the characteristic of $\KK$ is prime to $G$. The argument that $\bfield(G,V_\LL) \leq d \Rightarrow \bfield(G,V) \leq d$ generalizes the one given in  \cite[Theorem~4.1]{bandeira2017estimation} in the special case that $\KK=\RR$ and $\LL=\CC$, and the notations $P/Q$ and $L/M$ follow that work.
\end{remark}

In view of Lemma~\ref{lem:base-change}, we can replace $\KK$ with an algebraic extension without affecting $\bfield(G,V)$ or $\gfield(G,V)$. In particular, we may assume without loss of generality that $\KK$ contains $|G|$th roots of unity. In view of the standing assumption that $\Char \KK \nmid |G|$, we may assume they are distinct. Because $G$ is abelian, there is then a basis of $V$ on which $G$ acts diagonally; and the corresponding dual basis of coordinate functions also receives a diagonal action. 

In what follows, whenever we speak of the characters occurring in $V$ or $V^*$, even in statements whose scope includes arbitrary $\KK$ subject only to the non-modular hypothesis, the reader should understand us to  mean the characters that appear after base-changing to an extension of $\KK$ containing $|G|$th roots of unity.

We fix the following additional notation.

\begin{itemize}
    \item $\NN$ is the nonnegative integers (i.e., including $0$).
    \item $\widehat G := \Hom(G,\KK^\times)$ is the character group of $G$, written multiplicatively. (Since $\KK$ contains distinct $|G|$th roots of unity, $\Hom(G,\KK^\times)$ is the full character group.)
    \item $x_1,\dots,x_N\in V^*$ is a basis of coordinate functions on $V$ on which $G$ acts diagonally.
    \item $\chi_1,\dots,\chi_N\in \widehat G$ are the characters by which $G$ acts, respectively, on $x_1,\dots, x_N$; i.e., such that we have 
    \[
    gx_i = \chi_i(g)x_i
    \]
for all $g\in G$ and all $i\in \{1,\dots,N\}$.\footnote{A pedantic point is that because $x_1,\dots,x_m$ is a basis of $V^*$ and not $V$, the characters $\chi_1,\dots,\chi_N$ are not the irreducible components of $V$ but rather their inverses in $\widehat G$.}
    \item $\bfa := (a_1,\dots,a_N)\in \ZZ^N$ is an integer lattice point; other boldface letters such as $\bfb$ and $\bfc$ are used similarly.
    \item $\bfx^\bfa := x_1^{a_1}\dots x_N^{a_N}$ is the Laurent monomial with exponent vector $\bfa$.
    \item $\LM := \{\bfx^\bfa : \bfa\in \ZZ^N\}$ is the multiplicative group of Laurent monomials, viewed as a subgroup of $\KK(V)^\times$.
    \item Following \cite[Ch.~6]{bruns-herzog},
    \begin{align*}
        \log: \LM &\rightarrow \ZZ^N\\
        \bfx^\bfa &\mapsto \bfa
    \end{align*}
    is the canonical isomorphism;
    \begin{align*}
        \exp: \ZZ^N &\rightarrow \LM\\
        \bfa &\mapsto \bfx^\bfa
    \end{align*}
    is its inverse.
\end{itemize}

In what follows we tend to think of $\LM$ and $\ZZ^N$ as identified via these maps, although we retain the notational distinction for conceptual clarity, including the use of multiplicative notation in $\LM$ vs. additive notation in $\ZZ^N$. 
\begin{itemize}
    \item As a shorthand, given $\bfa\in \ZZ^N$, define the character
    \[\chi^\bfa := \chi_1^{a_1}\dots \chi_N^{a_N}\in \widehat G.\]
    \item Then we have a group homomorphism
\begin{align*}
\Theta: \LM &\rightarrow \widehat G \\
\bfx^\bfa&\mapsto \chi^\bfa,
\end{align*}
that, given a Laurent monomial, specifies the character by which $G$ acts on it.
\end{itemize}

The kernel of $\Theta$ consists of those monomials that are invariant under the action; in other words,
\[
\ker \Theta = \LM \cap \KK(V)^G.
\]
This kernel is identified via $\exp$ with a sublattice of $\ZZ^N\subset\RR^N$. We give the latter a name that shows the dependence on $G$ and the representation $V$:

\begin{itemize}
    \item Define the {\em lattice of the representation}
    \begin{align*}
    L(G,V) := &\log(\ker \Theta) \\
    = &\{ \bfa\in \ZZ^N : \chi^\bfa = 1 \in \widehat G\}.
\end{align*}
\end{itemize}

\begin{lemma}\label{lem:index}
If $V$ is a faithful representation of $G$, then $\Theta \circ \exp$ induces an isomorphism of the quotient group $\ZZ^N / L(G,V)$ with the character group $\widehat G$. In particular, the index of $L(G,V)$ in $\ZZ^N$ is $|G|$.
\end{lemma}

\begin{proof}
The character group of $\widehat G/\Theta(\LM)$ is the subgroup of the character group of $\widehat G$ that vanishes on all of $\Theta(\LM)$. Pontryagin duality identifies it with the kernel of the action of $G$ on $\LM$. Because $V$ is faithful, this is trivial. So $\widehat G / \Theta(\LM)$ has a trivial character group; thus it itself is trivial. In other words, $\Theta(\LM) = \widehat G$.

Since $\exp$ identifies $\ZZ^N$ with $\LM$, and $L(G,V)$ with the kernel of $\Theta$, we conclude that $\Theta \circ \exp$ induces an isomorphism of the quotient $\ZZ^N/L(G,V)$ with $\widehat G$, as claimed. Then
\[
[\ZZ^N: L(G,V)] = |\widehat G| = |G|
\]
follows.
\end{proof}

\begin{remark}
Although it is not explicitly drawn out in that work, Lemma~\ref{lem:index} is essentially proven over the course of the proof of Proposition~2.4 of  \cite{domokos}.
\end{remark}

\begin{remark}
Lemma~\ref{lem:index}'s conclusion about the index of $L(G,V)$ is an analogue to the basic Galois-theoretic fact that $|G| = [\KK(V):\KK(V)^G]$; an alternative proof of this conclusion starts from this fact and applies Lemma~\ref{lem:coset-basis} below (with $L = \ZZ^N$ and $L'=L(G,V)$).
\end{remark}

\begin{remark}
Lemma~\ref{lem:index} in particular implies that $L(G,V)$ is always a full-rank sublattice of $\ZZ^N$, since $|G|$ is finite. (Although the lemma requires the hypothesis that $G$ is faithful, this inference does not: if it is not faithful, replace it with its image in $\operatorname{Aut} V$, which can only be smaller.) 

Here and throughout, we have adopted the convention that the words {\em lattice} and {\em sublattice} (unadorned) indicate a  discrete subgroup $L$ of a Euclidean space $\mathbb{E}$ that is not necessarily full rank; the modifier {\em full-rank} is needed to imply that $\RR\otimes_\ZZ L \cong \mathbb{E}$.
\end{remark}

The case $G = \ZZ/n\ZZ$ ($n$ natural) and especially its subcase $G=\ZZ/p\ZZ$ ($p$ prime) are of particular concern to us. In these cases, the elements of $\widehat G$ can be represented by integers: we represent by $A\in \ZZ$ the character
\begin{equation}\label{eq:character-integer}
a\mapsto \zeta^{Aa},\; a\in G,
\end{equation}
where $\zeta$ is a fixed $n$th, respectively $p$th, root of unity in $\KK$. (The integer $A$ is determined mod $n$, respectively $p$.) Then the equation $\chi^\bfa = 1$ defining $L(G,V)$ can be written in the particularly simple form
\begin{equation}\label{eq:lattice-eq}
A_1a_1 + \dots + A_Na_N = 0 \pmod n,
\end{equation}
where
\begin{itemize}
    \item $A_1,\dots,A_N$ are {\em integers} representing  the characters $\chi_1,\dots,\chi_N$, as above.
\end{itemize}
Note that utilizing this convention requires us to switch to additive notation in the character group.

Points of the first orthant $\NN^N\subset \ZZ^N$ correspond via $\exp$ with bona-fide (non-Laurent) monomials. This gives them a notion of degree:

\begin{definition}\label{def:degree}
For $\bfa\in \NN^N$, the {\em degree} of $\bfa\in \ZZ^N$ is
\begin{align*}
    \deg \bfa := &a_1+\dots+a_N \\
    = &\deg \bfx^\bfa.
\end{align*}
\end{definition}

\begin{itemize}
    \item
Denote by $\Delta_N$ the convex hull in $\RR^N$ of $0$ and the standard basis vectors $\bfe_1,\dots,\bfe_N$. (Note that $\Delta_N$ is a closed simplex of volume $1/N!$.)
\end{itemize}

For any positive real number $d$, the integer points occurring in the dilation $d\Delta_N$ are precisely the points of degree $\leq d$. 

We now show that the determination of $\bfield(G,V)$, respectively $\gfield(G,V)$, is equivalent to the determination of the smallest $d$ such that $d\Delta_N$ contains a generating set for $L(G,V)$, respectively a full-rank sublattice of $L(G,V)$. This is the lattice-field equivalence promised at the beginning of the section, which forms the basis for our study of $\bfield$ and $\gfield$. 

\begin{remark}
Lemmas~\ref{lem:equivalence} and \ref{lem:equivalence2} which follow are in the spirit of many results in the invariant theory of abelian groups, and more generally in the theory of semigroup rings, which relate algebraic properties of a ring to more combinatorial properties of an underlying semigroup; see for example \cite[Chapter~6]{bruns-herzog}, \cite[Chapter~7]{miller2005combinatorial} for general results of this kind on semigroup rings, and \cite{schmid1991finite, finklea2008invariant, cziszter2013generalized, cziszter2016interplay, domokos, domokos2018syzygies} for a sampling of such results for invariant rings and separating sets. In \cite{hubert2013scaling, hubert-labahn}, this approach is used to study fields of rational invariants. 
\end{remark}

\begin{itemize}
    \item Following standard usage, $\langle S \rangle$ is the subgroup of a group $A$ generated by a subset $S\subset A$.
\end{itemize}

\begin{lemma}\label{lem:equivalence}
Let $d$ be a positive integer. 
\begin{enumerate}
    \item The points of $L(G,V)$ contained in $d\Delta_N$ generate $L(G,V)$ as a lattice if and only if $\KK[V]^G_{\leq d}$ generates $\KK(V)^G$ as a field.
    \item \label{statement:finite-degree} The points of $L(G,V)$ contained in $d\Delta_N$ generate a full-rank sublattice of $L(G,V)$ if and only if $\KK(V)^G$ is a finite field extension of the field generated by $\KK[V]^G_{\leq d}$. Furthermore, when these equivalent conditions hold, the degree of the field extension equals the index of the lattice containment, i.e.,
    \begin{equation}\label{eq:index-degree}
    [\KK(V)^G : \KK(\KK[V]^G_{\leq d})] = [L(G,V) : \langle L(G,V) \cap d\Delta_N\rangle].
    \end{equation}
\end{enumerate}
\end{lemma}

\begin{proof}
The first statement is a consequence of the second---it is the statement that if one of the equivalent conditions in the second statement holds, and one side of equation~\eqref{eq:index-degree} is $1$, then the other (equivalent) condition holds and the other side is also $1$. So it suffices to prove the second statement.

Monomials are algebraically independent if and only if their exponent vectors (i.e., their $\exp$-preimages) are linearly independent. Also, the monomials in $\KK[V]^G_{\leq d}$, which are the $\exp$-image of $d\Delta_N \cap L(G,V)$, form a vector space basis for it, so $\KK[V]^G_{\leq d}$ contains $N$ algebraically independent elements if and only if it contains that many algebraically independent monomials. Putting this together, it follows that $d\Delta_N$ contains $N$ linearly independent points of $L(G,V)$ if and only if $\KK[V]^G_{\leq d}$ contains $N$ algebraically independent elements. In other words, the field
\[
\KK(\KK[V]^G_{\leq d})
\]
has full transcendence degree $N=\operatorname{tr.deg}\KK(V)^G = \operatorname{tr.deg}\KK(V)$ if and only if the lattice
\[
\langle L(G,V) \cap d\Delta_N\rangle
\]
has full rank $N=\operatorname{rk}L(G,V)$. Now $\KK(V)^G$ is a finitely generated field extension of $\KK$ (generated for example by a set of algebra generators of $\KK[V]^G$ over $\KK$), thus when $\KK(\KK[V]^G_{\leq d})$ has full transcendence degree, $\KK(V)^G$ is actually a finite field extension of it. So the first part of statement~\ref{statement:finite-degree} is proven. 

It remains to show the equality~\eqref{eq:index-degree}. We do this by verifying that, when $\langle L(G,V) \cap d\Delta_N \rangle$ is full-rank, a set of coset representatives for $\langle L(G,V) \cap d\Delta_N\rangle$ in $L(G,V)$ corresponds via $\exp$ to a vector space basis for $\KK(V)^G$ over $\KK(\KK[V]^G_{\leq d})$. Let 
\[
L := \exp(L(G,V))\subset \LM
\]
and let 
\[
L' := \exp(\langle L(G,V) \cap d\Delta_N\rangle)\subset L.
\]
Then the group algebras $\KK[L]$ and $\KK[L']$ may be viewed as subrings of $\KK(V)$, and we have
\[
\KK[V]^G \subset \KK[L]\subset \KK(V)^G
\]
and
\[
\KK[\KK[V]^G_{\leq d}] \subset \KK[L']\subset \KK(\KK[V]^G_{\leq d}).
\]
Taking fraction fields, we get 
\[
\KK(L):=\Frac \KK[L] = \KK(V)^G
\]
and 
\[
\KK(L'):=\Frac \KK[L']=\KK(\KK[V]^G_{\leq d}).
\]
Thus the task is to show that a set of coset representatives for $L'$ in $L$ is a vector space basis for $\KK(L)$ over $\KK(L')$. This is the content of the following lemma, which will complete the proof.
\end{proof}

\begin{lemma}\label{lem:coset-basis}
Let $L'\subset L \subset\LM$ be subgroups of $\LM$ of the same rank, and let $\KK(L')$ and $\KK(L)$ be the fraction fields of their group algebras, viewed as subfields of $\KK(V)$ as above. 

Then a complete set of coset representatives for $L'$ in $L$ forms a vector space basis for $\KK(L)$ over $\KK(L')$.
\end{lemma}

\begin{proof}
The assumption about ranks implies the index of $L'$ in $L$ is finite. Let $m_1,\dots,m_r$ be a complete set of coset representatives. Note that they form a free module basis for $\KK[L]$ over $\KK[L']$. This already implies that they are linearly independent over $\KK(L')$, by clearing denominators in a hypothesized linear relation over $\KK(L')$ to obtain one over $\KK[L']$, contradiction. We need to show they span $\KK(L)$ over $\KK(L')$. We do this by showing that $\KK(L)$ is the field generated over $\KK(L')$ by $m_1,\dots,m_r$, and that this is no bigger than the vector space generated over $\KK(L')$ by $m_1,\dots,m_r$.

The field generated over $\KK(L')$ by $m_1,\dots,m_r$ contains every coset of $L'$ in $L$. Thus it contains $L$, thus $\KK[L]$, thus $\KK(L)$ (and is equal to the last of these). Now because the quotient group $L/L'$ has finite order $r$, we have $m_i^r\in L'\subset \KK(L')$ for each $i$. In particular, each $m_i$ is algebraic over $\KK(L')$. Thus the field generated over $\KK(L')$ by $m_1,\dots,m_r$ is no bigger than the ring $\KK(L')[m_1,\dots,m_r]$ generated by them. On the other hand, this ring is no bigger than the {\em module} generated over $\KK(L')$ by $m_1,\dots,m_r$ since any product of $m_i$'s is contained in $\KK[L]$, which is already generated as a module over $\KK[L']\subset \KK(L')$ by the $m_i$'s. Putting all this together, we have
\begin{align*}
    \KK(L) &= \KK(L')(m_1,\dots,m_r)\\
    &= \KK(L')[m_1,\dots,m_r]\\
    &= \KK(L')m_1 + \dots + \KK(L')m_r,
\end{align*}
so we conclude that the $m_i$'s span $\KK(L)$ over $\KK(L')$. This completes the proof.
\end{proof}

The following is an immediate corollary of Lemma~\ref{lem:equivalence}.
\begin{lemma}\label{lem:equivalence2}
We have
\[
\bfield(G,V) = \min(d : L(G,V) = \langle L(G,V) \cap d\Delta_N\rangle)
\]
and
\[
\gfield(G,V) = \min(d : \operatorname{rk}\langle L(G,V)\cap d\Delta_N\rangle =N).
\]
\end{lemma}

\begin{proof}
Each equality follows from the corresponding statement in Lemma~\ref{lem:equivalence} in view of the definitions of $\bfield$ and $\gfield$.
\end{proof}

We introduce terminology for the expressions on the right side of the equations in Lemma~\ref{lem:equivalence2}:

\begin{definition}\label{def:generated-in-degree}
For a lattice $L\subset\ZZ^N$ and a natural number $d$, if $L = \langle L\cap d\Delta_N\rangle$ then we say that $L$ is {\em generated in degree $\leq d$}. If $d$ is the minimum natural number such that $L$ is generated in degree $\leq d$, we say that $L$ is {\em generated in degree $d$}, and refer to $d$ as the {\em generation degree} of $L$.
\end{definition}

\begin{definition}\label{def:full-rank-degree}
For a lattice $L\subset\ZZ^N$ and a natural number $d$, if $\operatorname{rk} L = \operatorname{rk} \langle L\cap d\Delta_N\rangle$, i.e., if $L$ has a full-rank sublattice generated in degree $\leq d$, then we say that $L$ has {\em full-rank degree $\leq d$}. If $d$ is the minimum natural number such that $L$ has full-rank degree $\leq d$, then we say that $L$ has {\em full-rank degree $d$}.
\end{definition}

In the rest of this work, we use the equivalence given by Lemma~\ref{lem:equivalence2} freely, often without explicit comment.

\subsection{Only the set of distinct nontrivial characters matters}

Having reduced the study of $\bfield$ and $\gfield$ to questions about the lattices $L(G,V)$, we now show that $\bfield$ and $\gfield$ are controlled entirely by the set of distinct, nontrivial characters of $G$ in the representation $V$ (and not their multiplicities). 

\begin{remark}
The results of this subsection, and in particular Lemma~\ref{lem:merge-identical}, are in the spirit of \cite[Proposition~4.7]{cziszter2016interplay}, \cite[Corollary~2.6]{domokos}, and \cite[Section~4]{domokos2018syzygies}, which relate the invariant ring of a representation of an abelian group to the invariant ring of a corresponding multiplicity-free representation. The proofs are related as well---in particular, the map $\pi'$ in the proof of Lemma~\ref{lem:merge-identical} is very nearly the {\em transfer homomorphism} appearing in \cite{cziszter2016interplay} and \cite{domokos2018syzygies}, and serves the same function, while the point $\mathbf w$ constructed in the proof of the same lemma below plays the same role as a similar point constructed in the proof of \cite[Corollary~2.6]{domokos} (there called $n$). Again, the setting and the precise goals differ, so we give self-contained proofs.
\end{remark}

\begin{lemma}\label{lem:delete-trivial}
Let $V'$ be a representation of $G$ obtained from $V$ by deleting a trivial character. Then $\bfield(G,V)=\bfield(G,V')$ and $\gfield(G,V)=\gfield(G,V')$.
\end{lemma}

In preparation for this and many proofs that follow, we draw out an elementary principle that will be used repeatedly:

\begin{observation}\label{obs:generate-over}
Suppose $\varphi:L\rightarrow M$ is a group homomorphism and $S\subset L$ is a subset. If $S$ contains generators for $\ker \varphi$ and $\varphi(S)$ contains generators for $\operatorname{im}\varphi$, then $S$ generates $L$.\qed
\end{observation}

\begin{proof}[Proof of Lemma~\ref{lem:delete-trivial}]
By Lemma~\ref{lem:equivalence2}, what we have to show is that $L(G,V)$ and $L(G,V')$ are generated in the same degree, and have the same full-rank degree.

Deleting a character from $V$ deletes the corresponding character from $V^*$. Without loss of generality we may assume it is $\chi_N$ that is the trivial character to be deleted. So
\[
\chi_1^{a_1}\dots\chi_N^{a_N} = 1\in \widehat G
\]
if and only if 
\[
\chi_1^{a_1}\dots\chi_{N-1}^{a_{N-1}} = 1\in \widehat G.
\]
In other words, $(a_1,\dots,a_N)\in L(G,V)$ if and only if $(a_1,\dots,a_{N-1})\in L(G,V')$, so that
\[
L(G,V)\cong L(G,V')\times \ZZ.
\]

Let $I : \RR^{N-1}\rightarrow\RR^N$ be the inclusion given by
\[
(a_1,\dots,a_{m-1})\mapsto (a_1,\dots,a_{m-1},0).
\]
For any $d\geq 1$, this embeds $d\Delta_{N-1}$ into $d\Delta_N$ and $L(G,V')$ into $L(G,V)$. Thus if $d\Delta_{N-1}$ contains a generating set for $L(G,V')$, then the latter's image under $I$ is contained in $d\Delta_N$. Furthermore, this image, together with $\bfe_N$ (which is automatically in $d\Delta_N$, as well as in $L(G,V)$ because $\chi_N$ is trivial), generate $L(G,V)$: this follows from  Observation~\ref{obs:generate-over} applied to the projection $\varphi:L(G,V)\cong L(G,V')\times \ZZ \rightarrow \ZZ$ to the final coordinate, because the kernel of this projection is $I(L(G,V'))$ and the image is generated by the image of $\bfe_N$. To summarize, if $L(G,V')$ is generated in degree $\leq d$, so is $L(G,V)$. 

In the other direction, the projection $\pi:\RR^N\rightarrow\RR^{N-1}$ to all but the last coordinate is a group homomorphism that maps $d\Delta_N$ onto $d\Delta_{N-1}$. Furthermore, it maps $L(G,V)$ surjectively onto $L(G,V')$, as can be seen from the fact that $\pi\circ I$ is the identity on $L(G,V')$. Thus if $L(G,V)\cap d\Delta_N$ generates $L(G,V)$, its image under $\pi$ generates $L(G,V')$ and is contained in $d\Delta_{N-1}$. So if $L(G,V)$ is generated in degree $\leq d$, then so is $L(G,V')$. We can conclude $L(G,V)$ and $L(G,V')$ have the same generation degree.

The exact same arguments, just replacing lattice generators everywhere with generators for full-rank sublattices, show that $L(G,V)$ and $L(G,V')$ have the same full-rank degree. This concludes the proof. 
\end{proof}

\begin{lemma}\label{lem:merge-identical}
Let $V'$ be a representation of $G$ obtained from $V$ by merging a pair of identical characters. Then $\bfield(G,V)=\bfield(G,V')$ and $\gfield(G,V)=\gfield(G,V')$.
\end{lemma}

\begin{proof}
The argument is similar to Lemma~\ref{lem:delete-trivial}. Without loss of generality we suppose the identical characters to be merged are $\chi_{N-1}$ and $\chi_N$. 

For the direction $\bfield(G,V')\leq d \Rightarrow \bfield(G,V)\leq d$ (and the corresponding statement for $\gfield$), we use the same inclusion $I:\RR^{N-1}\rightarrow \RR^N$ defined in the proof of Lemma~\ref{lem:delete-trivial}. Suppose that $d\Delta_{N-1}$ contains a generating set for $L(G,V')$. Then it in particular must contain points with nonvanishing $a_{N-1}$, since $L(G,V')$ is full-rank in $\RR^{N-1}$. Fix one such point $(a_1,\dots,a_{N-1})$; because it is in $d\Delta_{N-1}$, $a_{N-1}$ must be positive. Then
\[
\mathbf{w} := (a_1,\dots,a_{N-1}-1,1)
\]
is contained in $d\Delta_N$, and is a point of $L(G,V)$ because $\chi_{N-1}=\chi_N$. Thus the set
\[
I(L(G,V')\cap d\Delta_{N-1}) \cup \{\mathbf{w}\}
\]
is contained in $d\Delta_N$. Furthermore, it is a generating set for $L(G,V)$, by Observation~\ref{obs:generate-over} applied to the projection to the last coordinate, since $I(L(G,V')\cap d\Delta_{N-1})$ generates the kernel of this projection while $\mathbf{w}$ generates the image. Thus we can conclude that if $L(G,V')$ has generation degree $\leq d$, so does $L(G,V)$.

In the other direction, we consider the map $\pi':\RR^N\rightarrow\RR^{N-1}$ given by
\[
(a_1,\dots,a_{N-2},a_{N-1},a_N)\mapsto (a_1,\dots,a_{N-2},a_{N-1}+a_N).
\]
As in the proof of Lemma~\ref{lem:delete-trivial}, this map is a group homomorphism, and it maps $d\Delta_N$ onto $d\Delta_{N-1}$ and $L(G,V)$ onto $L(G,V')$ (the latter because $\pi'\circ I$ is the identity on $L(G,V')$). So the same logic as in the proof of Lemma~\ref{lem:delete-trivial} shows that if $L(G,V)$ is generated in degree $\leq d$, so is $L(G,V')$. Thus $L(G,V)$ and $L(G,V')$ have the same generation degree.

Again, the exact same arguments, with generating sets replaced by generators for full-rank sublattices, show that $L(G,V)$ and $L(G,V')$ have the same full-rank degree. This completes the proof.
\end{proof}

By induction, the following is an immediate corollary:

\begin{lemma}\label{lem:distinct-nontrivial}
The numbers $\bfield(G,V)$ and $\gfield(G,V)$ depend only on the set of distinct, nontrivial characters of $G$ that occur in $V$ (and not on their multiplicities). \qed
\end{lemma}

In view of Lemma~\ref{lem:distinct-nontrivial} (and the fact that the characters appearing in the definition of $L(G,V)$ are those belonging to $V^*$ rather than $V$), we introduce the following notation:
\begin{itemize}
    \item $\Supp V$ is the set of distinct, nontrivial characters in $V^*$.
    \item $m$ is the cardinality of $\Supp V$. Note: below, this notation often occurs in statements of results in which we do not assume $\KK$ contains enough roots of unity to diagonalize the action of $G$ on $V$. As mentioned above, in such cases, $m$ should be understood to mean the cardinality of $\Supp V_\LL$, where $\LL$ is an extension of $\KK$ that does diagonalize the action.
    \item If $S\subset\widehat G\setminus\{1\}$ is any set of distinct, nontrivial characters of $G$, then $\ZZ^S$ is the free abelian group with basis $\{\bfe_\chi\}_{\chi\in S}$ indexed by the elements of $S$. We represent an element $\bfa\in \ZZ^S$ as a tuple $(a_\chi)_{\chi\in S}$ of integers, shorthand for $\sum_{\chi\in S} a_\chi\bfe_\chi$. 
    \item We view $\ZZ^S$ as a lattice; it is identified with the integer lattice $\ZZ^{|S|}\subset \RR^{|S|}$ up to permutation of the axes. Thus the degree of a point in the nonnegative orthant $\{\bfa : a_\chi \geq 0\text{ for all }\chi \in S\}$ is $\sum_{\chi \in S} a_\chi$, as in Definition~\ref{def:degree}.
    \item There is a natural homomomorphism $\ZZ^S\rightarrow\widehat G$ given by 
    \[
    \bfa \mapsto \prod_{\chi\in S} \chi^{a_\chi}.
    \]
    We denote by $L(G,S)\subset \ZZ^S$ the kernel of this homomorphism.
    \item $\bfield(G,S)$ and $\gfield(G,S)$ are the generation degree and full rank degree of $L(G,S)$, respectively.
    \item In particular, $L(G,\Supp V)$ is the lattice of a representation $V'$ of $G$ obtained by deleting all trivial characters from $V$ and replacing each remaining  isotypic component with just one copy of the corresponding character.
\end{itemize}
With this notation, the content of Lemma~\ref{lem:distinct-nontrivial} is the equalities $\bfield(G,V) = \bfield(G,\Supp V)$ and $\gfield(G,V) = \gfield(G, \Supp V)$. The prime in the notation $\Supp$ is to acknowledge the differences with the standard notion of ``support": dualization and deletion of the trivial character (if it appears).

If the map $G\rightarrow GL(V)$ defining the action of $G$ on $V$ is precomposed with an automorphism of $G$, then the characters in $V$ (and thus those in $V^*$) are affected; however, the ring $\KK[V]^G$, the field $\KK(V)^G$, and the lattice $L(G,V)$ are unaffected. Thus $\bfield(G,V)$ and $\gfield(G,V)$ are unaffected.\footnote{For example, the negation map is an automorphism of $G$, and it induces the negation map on $\widehat G$, so $L(G,V) = L(G,V^*)$. In particular, the care we have taken to distinguish the characters of $V$ from those of $V^*$ is for conceptual clarity, not because of an effect on $\bfield$ or $\gfield$.} In view of this, we introduce one last notation:
\begin{itemize}
    \item If $S\subset \widehat G\setminus\{0\}$ is a set of nontrivial characters, $[S]$ is the orbit of $S$ under the natural action of $\Aut G$ on subsets of $\widehat G$.
\end{itemize}
Identifying $\ZZ^S$ with $\ZZ^{|S|}$, which is a well-defined identification up to the order of the axes, we observe that $L(G,[S])\subset \ZZ^{|S|}$ is then well-defined, up to the same ambiguity in the order of the axes. Thus $\bfield(G,[S])$ and $\gfield(G,[S])$ are well-defined. 

\section{General results}\label{sec:general}

In this section we prove a general lower bound on $\gfield(G,V)$ (and thus $\bfield(G,V)$), and an upper bound on $\bfield(G,V)$ in the special case $G=\ZZ/p\ZZ$ for a prime number $p$. Both bounds depend on the order of $G$ and the number $m$ of distinct nontrivial characters in $V$. We also show that a similar (but in general weaker) lower bound holds at the generality of finite groups and fields of arbitrary characteristic, depending on $N=\dim_\KK V$ rather than $m$.

The lower bound, Theorem~\ref{thm:lower-bound}, is sharp. The upper bound for $G=\ZZ/p\ZZ$, Theorem~\ref{thm:upper-bound}, is not sharp; nonetheless it improves on the Noether bound in most cases. The Noether bound is attained by $\beta(G,V)$ for any nontrivial representation of $G=\ZZ/p\ZZ$. When $\KK$ is algebraically closed, the same is true of $\bsep(G,V)$, as follows from \cite[Theorem~2.1]{domokos} (and actually the argument works as long as $\KK$ contains $p$th roots of unity).  With an eye to the signal processing application discussed in the introduction, we verify in this section that $\bsep(G,V)$ never drops below the Noether bound in the case that $\KK=\RR$ either (Proposition~\ref{prop:separating-over-R}). Thus we establish a gap  between $\bfield$ and $\beta, \bsep$ for $G=\ZZ/p\ZZ$ in these cases. We conjecture a sharp upper bound on $\bfield$ below in Section~\ref{sec:open-questions}.

We also include some related results. In Proposition~\ref{prop:lower-bound-attained}, we characterize the groups and representations that attain the lower bound. As an artifact of the proof of the lower bound, we obtain (Proposition~\ref{prop:b=g-when-close-to-bottom}) that if $\gfield(G,V)$ is sufficiently close to the bound, then $\gfield(G,V)=\bfield(G,V)$. When $m$ is large, the lower bound becomes very low, but we show (Proposition~\ref{prop:gfield-at-least-3}) that under mild hypotheses, $\gfield(G,V)$ still does not go below $3$. Meanwhile, the upper bound for $G=\ZZ/p\ZZ$ is proven by bootstrapping from the special case $m=2$ (which is studied in more detail in the next section). The induction step depends on a result (Proposition~\ref{prop:subsets}) which relates $\bfield(\ZZ/p\ZZ,S)$ to $\bfield(\ZZ/p\ZZ,S')$ for certain subsets $S'\subset S\subset \widehat G\setminus\{0\}$, which is of independent interest. It can be informally summarized as stating that $\bfield(\ZZ/p\ZZ,S)$ is not too far from being a nonincreasing function of $S$ with respect to set containment.

Although our main goal is the study of $\bfield$ and $\gfield$, the above-mentioned verification that the Noether bound is always attained for $\bsep$ for faithful representations of $\ZZ/p\ZZ$ over $\RR$ involves a step (Lemma~\ref{lem:multigraded-separating}) that might be of use to those who study degree bounds for separating sets; it is a straightforward generalization of \cite[Lemma~2.5]{domokos}.

Except for Theorem~\ref{thm:lower-bound-general}, Lemma~\ref{lem:multigraded-separating}, Proposition~\ref{prop:separating-over-R}, and part of Proposition~\ref{prop:lower-bound-attained}, all proofs in this section begin by applying Lemmas~\ref{lem:equivalence2} and \ref{lem:distinct-nontrivial} to identify $\bfield(G,V)$ and $\gfield(G,V)$ with the generation degree and full-rank degree, respectively, of $L(G,\Supp V)$. To avoid repetitiveness, we make this reduction without comment going forward (except as otherwise noted).

\subsection{Lower bounds for general $G$ and related results}

\begin{theorem}\label{thm:lower-bound}
If $G$ is a finite abelian group, and $V$ is a faithful, non-modular, finite-dimensional representation of $G$, and $m$ is the number of distinct, nontrivial characters occurring in $V$, then
\[
\gfield(G,V) \geq \sqrt[m]{|G|}.
\]
\end{theorem}

\begin{proof}
It is convenient to fix an order for $\chi_1,\dots,\chi_m\in \Supp V$, so as to identify $\ZZ^{\Supp V}$ with $\ZZ^m$. Let $\bfa_1,\dots,\bfa_m$ be a set of linearly independent points of $L(G,\Supp V)$ lying in the first orthant $\NN^m$, and satisfying $\deg(\bfa_i)\leq \gfield(G,V)$ for all $i$. Let 
\[
T := \{\alpha_1\bfa_1+\dots+\alpha_m\bfa_m : 0 \leq \alpha_i < 1\} \subset \RR^m
\]
be the parallelotope spanned by $\bfa_1,\dots,\bfa_m$ (which is a fundamental parallelotope of the sublattice they generate), and let $|\bfa_i|$ be the Euclidean norm.

We argue that 
\[
|G| \leq \Vol(T) \leq |\bfa_1|\dots|\bfa_m| \leq \deg(\bfa_1)\dots\deg(\bfa_m) \leq \gfield(G,V)^m,
\]
whereupon the theorem follows by taking $m$th roots. First, if $\bfa_1,\dots,\bfa_m$ generate $L(G,\Supp V)$ then the volume of $T$ is equal to $|G|$ by Lemma~\ref{lem:index} and the fact that the index of a full-rank lattice in $\ZZ^m$ is equal to the volume of its fundamental parallelotope. Otherwise, $\bfa_1,\dots,\bfa_m$ generate a proper (but still full-rank) sublattice, so its index, and hence the volume of $T$, is greater. This establishes the first inequality. The second inequality is Hadamard's inequality. (Equality occurs if and only if $\bfa_1,\dots,\bfa_m$ are pairwise orthogonal.) The third inequality holds because $|\bfa_i|\leq \deg(\bfa_i)$ for each $i$ by the triangle inequality. The last inequality is because $\deg(\bfa_i)\leq \gfield(G,V)$ for each $i$ by construction.
\end{proof}

At the price of replacing $m$ with $N = \dim_\KK V$, we can remove the  restrictions to abelian groups and coprime characteristics:

\begin{theorem}\label{thm:lower-bound-general}
If $G$ is a finite (not necessarily abelian) group, and $V$ is a faithful (not necessarily non-modular) representation of $G$ of finite dimension $N$, then
\[
\gfield(G,V) \geq \sqrt[N]{|G|}.
\]
\end{theorem}

\begin{proof}
The proof proceeds as in Theorem~\ref{thm:lower-bound} except working in the original ring $\KK[V]^G$ instead of a lattice, and with \cite[Theorem~4]{mittmann2014algebraic} taking the place of Hadamard's inequality and the triangle inequality. Let $f_1,\dots,f_N$ be $N$ algebraically independent elements of $\KK[V]^G$, ordered in increasing degree order, and chosen so as to minimize the maximum degree $\deg(f_N)$. We have
\[
[\KK(V):\KK(f_1,\dots,f_N)] \leq \deg(f_1)\dots\deg(f_N) \leq \deg(f_N)^N.
\]
where the first inequality is \cite[Theorem~4]{mittmann2014algebraic}. Since $\KK(f_1,\dots,f_N)\subset \KK(V)^G$, it follows that
\[
|G| = [\KK(V):\KK(V)^G] \leq [\KK(V):\KK(f_1,\dots,f_N)] \leq \deg(f_N)^N.
\]
Because $\deg(f_N) = \gfield(G,V)$ by our choice of $f_1,\dots,f_N$, the result follows by taking $N$th roots.
\end{proof}

\begin{remark}
The principal ingredient of Theorem~\ref{thm:lower-bound-general} is \cite[Theorem~4]{mittmann2014algebraic}, which itself is just a dehomogenized version of \cite[Corollary~1.8]{kemper1996constructive}. The latter is proven via intersection theory. An alternative proof of Theorem~\ref{thm:lower-bound} would be to combine Theorem~\ref{thm:lower-bound-general} with Lemma~\ref{lem:distinct-nontrivial}. We give the geometry of numbers-style proof above because it can be adapted uneventfully to the situation considered in \cite{bbs-rational}, and because it is interesting and suggestive that it and the intersection-theoretic proof of Theorem~\ref{thm:lower-bound-general} are getting at the same thing. To further illustrate the latter point: it follows from \cite[Corollary~1.8]{kemper1996constructive} that if the $f_i$ are homogeneous and realize equality in the inequality $|G|\leq \prod \deg(f_i)$, then the $f_i$ necessarily generate the invariant {\em ring} $\KK[V]^G$ (not just the field). The intersection-theoretic argument in \cite{kemper1996constructive} involves deducing from the equality $|G|=\prod \deg(f_i)$ that a certain intersection of projective hypersurfaces (over the algebraic closure of a rational function field over $\KK$) is empty, applying the Nullstellensatz, and then reasoning about integrality. In the abelian, coprime characteristic case, we can see the same result in the Euclidean geometry discussed in the above proof of Theorem~\ref{thm:lower-bound} (provided we work in $L(G,V)$ rather than $L(G,\Supp V)$). Equality in $|G|\leq \prod \deg(\bfa_i)$ implies the $\bfa_i$'s generate $L(G,V)$, and also forces equality in both Hadamard's inequality and the triangle inequality, thus the $\bfa_i$ are mutually orthogonal and $|\bfa_i|=\deg(\bfa_i)$ for all $i$. Either of these conclusions implies that each $\bfa_i$ lies on a coordinate axis. Therefore the $\bfa_i$ generate the first orthant as a simplicial cone. Since they also generate $L(G,V)$ as a group, it follows that they generate the semigroup $L(G,V)\cap \NN^N$, and therefore that the $\bfx^{\bfa_i}$ generate the semigroup ring $\KK[V]^G$.
\end{remark}

The bounds in Theorems~\ref{thm:lower-bound} and \ref{thm:lower-bound-general} are sharp for any $m$, respectively $N$:

\begin{example}\label{ex:lower-bound-attained}
Fixing natural numbers $m=N\geq 1$ and $d\geq 2$, a field $\KK$ containing distinct $d$th roots of unity, and a faithful character $\chi:\ZZ/d\ZZ\rightarrow \KK^\times$, the bounds in Theorems~\ref{thm:lower-bound} and \ref{thm:lower-bound-general} are attained by the group $G=(\ZZ/d\ZZ)^m$, acting on $V=\KK^m$ by the $m$ characters $\chi_j$, $j=1,\dots, m$ obtained from $\chi$ composed with projecting $G$ to the $j$th factor (which form a basis for the character group of $G$ as a $(\ZZ/d\ZZ)$-module). If $x_1,\dots,x_m$ is the basis for $V^*$ dual to the diagonal basis for the action (as in Section~\ref{sec:fields-and-lattices}), then the invariant ring $\KK[V]^G$ is generated by the $m$ monomials $x_j^d$, and there are no nonconstant invariants of degree $<d$, so 
\[
\gfield(G,V)=\bfield(G,V) = d = \sqrt[m]{|G|}.
\]
\end{example}

If $G$ is abelian, or if $\KK=\CC$, then  Example~\ref{ex:lower-bound-attained} is essentially the only way that these bounds can be attained:

\begin{proposition}\label{prop:lower-bound-attained}
In Theorem~\ref{thm:lower-bound}, if equality is attained, then $\KK$ contains $\bfield(G,V)$th roots of unity, and $G$ and $V$ are, after dropping trivial characters and duplicate characters, the $G$ and $V$ of Example~\ref{ex:lower-bound-attained} up to isomorphisms of each of them.

In Theorem~\ref{thm:lower-bound-general}, if equality is attained and also $\KK=\CC$, then $G$ and $V$ are, up to isomorphisms, the $G$ and $V$ of Example~\ref{ex:lower-bound-attained}.
\end{proposition}

\begin{proof}
Let $d:= \gfield(G,V)$.  Considering the abelian case first, suppose there is equality in Theorem~\ref{thm:lower-bound}. We temporarily base change to a ground field $\tilde \KK$ containing $d$th roots of unity in order to apply our lattice methods; by Lemma~\ref{lem:base-change} this does not affect the hypothesis, and once we know enough about $G$ and $V$, it will be clear that $\KK$ must have contained $d$th roots of unity to begin with. Tracing through the proof of Theorem~\ref{thm:lower-bound}, equality requires that all of the following hold:
\begin{enumerate}
    \item $|G|=\Vol(T)$, so $\bfa_1,\dots,\bfa_m$ are generators for $L(G,\Supp V)$.\label{pro:they-generate}
    \item $\Vol(T) = |\bfa_1|\dots |\bfa_m|$, so the $\bfa_i$ are mutually orthogonal. Since they are in $\NN^m$, it follows that they lie on the coordinate axes.\label{pro:orthogonal}
    \item $|\bfa_1|\dots |\bfa_m|=d^m$ while each $|\bfa_i|\leq d$, so (in view of \ref{pro:orthogonal}) actually each $\bfa_i = d\bfe_j$ for a different $j$. (We can permute the $\bfa_i$ to say $\bfa_i=d\bfe_i$, if desired.)\label{pro:what-they-are}
\end{enumerate}
Combining \ref{pro:they-generate} with \ref{pro:what-they-are}, we see $L(G,\Supp V)$ is the lattice $d\ZZ^m$. So Lemma~\ref{lem:index}, applied to a representation $V'$ made from $V$ by dropping trivial and duplicate characters, shows that
\[
\widehat G \cong \ZZ^m / d\ZZ^m \cong (\ZZ/d\ZZ)^m,
\]
and therefore $G\cong (\ZZ/d\ZZ)^m$ too. Because $L(G,V')$ is the kernel of the map $\Theta$ of Section~\ref{sec:fields-and-lattices}  (describing the action of $G$ on the Laurent monomials written in the diagonal basis), and, for each $i=1,\dots,m$, it does not contain $j\bfe_i$ for $j=1,\dots,d-1$, the character $\chi_i=\Theta(x_i)$ describing the action of $G$ on $x_i$ must factor through a {\em faithful} character of $\ZZ/d\ZZ$. Because the $x_i$ generate $\mathcal{LM}$, the $\chi_i$ ($i=1,\dots,m$) generate $\widehat G$ (again by Lemma~\ref{lem:index}). Since there are $m$ of them, it follows that they are a $\ZZ/d\ZZ$-basis for $\widehat G$. Fix any primitive $d$th root of unity $\zeta \in \KK$; then $G$ has a $\ZZ/d\ZZ$-basis $e_1,\dots,e_m$ dual to $\chi_1,\dots,\chi_m$ in the sense that $\chi_i(e_j) = \zeta^{\delta_{ij}}$ for all $1\leq i,j\leq m$ (where $\delta_{ij}$ is the Kronecker delta). Writing elements of $G$ on the basis $e_1,\dots,e_m$ yields an isomorphism of $G$ with the $G$ of Example~\ref{ex:lower-bound-attained}, in such a way that $V'$ is isomorphic with the $V$ of Example~\ref{ex:lower-bound-attained} as well. 

It remains to verify that the original field of definition $\KK$ for the original representation $V$ (before the base change and the deletion of trivial and duplicate characters) must have contained $d$th roots of unity all along. Let $\rho: G\rightarrow GL(V)$ be the original representation map (defined over $\KK$). Define a projection $\pi:V\rightarrow V$ by
\[
\pi := \frac{1}{d}\sum_{j=0}^{d-1} \rho(e_1)^j
\]
where $e_1$ is as in the previous paragraph. Note that $\pi$ is defined over (the original) $\KK$. The kernel of $\pi$ is the isotypic component of $\chi_1^{-1}$; it is nontrivial over $\tilde \KK$, and defined over $\KK$, so it is a nontrivial subspace of the original $V$. Any nonzero element in $\ker \pi$ is an eigenvector for  $\rho(e_1)$ with eigenvalue $\zeta^{-1}$, so we must have $\zeta\in \KK$. This completes the verification that $\KK$ contained $d$th roots of unity the whole time.

Now we consider the case where $\KK=\CC$ but $G$ may be nonabelian.  In the situation of Theorem~\ref{thm:lower-bound-general}, there exist $N$ algebraically independent elements $f_1,\dots,f_N\in \KK[V]^G$ that realize the bound $\deg(f_i)\leq d:=\bfield(G,V)$. We can assume the $f_i$ are homogeneous; if not, split them into homogeneous components, whereupon some subset of $N$ of the homogeneous components must be algebraically independent, and use these $N$ homogeneous components as the $f_i$ instead. If we have equality in Theorem~\ref{thm:lower-bound-general}, then, tracing through the proof, we see that
\[
|G| = [\CC(V):\CC(V)^G] = [\CC(V):\CC(f_1,\dots,f_N)] = \prod_1^N \deg(f_i)=d^N
\]
The second and third equalities imply by \cite[Corollary~1.8]{kemper1996constructive} that $f_1,\dots,f_N$ generate $\CC[V]^G = \CC(V)^G\cap \CC[V]$ as an algebra. Because $\CC[V]^G$ has Krull dimension $N$ (as the polynomial ring $\CC[V]$ is integral over it), it is thus a polynomial algebra. Therefore $(G,V)$ is a unitary reflection group, by the Chevalley-Shepard-Todd theorem, and the degrees of the $f_i$ are uniquely determined by $(G,V)$. Such a group is always a direct product of irreducible unitary reflection groups, acting in orthogonal spaces \cite[Theorem~1.27]{lehrer2009unitary}. Thus $\CC[V]^G$ is a tensor product of invariant rings of irreducible unitary reflection groups, and the degrees of the $f_i$ are obtained by amalgamating the degrees of the fundamental invariants of the irreducible components. Now the fourth equality above implies, in view of the corresponding inequality in the proof of Theorem~\ref{thm:lower-bound-general}, that $\deg f_i = d$ for all $i=1,\dots, N$, while any irreducible unitary reflection group acting in a space of dimension at least 2 has fundamental invariants of at least 2 distinct degrees (e.g., by \cite[Appendix~D.2]{lehrer2009unitary}). It follows that all irreducible components of $G$ are one-dimensional, with a fundamental invariant of degree $d$. The irreducible complex reflection group acting in a 1-dimensional space with a degree-$d$ fundamental invariant is $\ZZ/d\ZZ$ acting by a faithful character. Thus $G$ is the direct product of $m$ of these. By automorphing the factors if needed, we can ensure they each act by the same faithful character. This yields the group and representation of Example~\ref{ex:lower-bound-attained}.
\end{proof}

We move back to the setting of abelian groups and (faithful) non-modular representations, first focusing on the special case $G=\ZZ/p\ZZ$. When $m=1$, this is an instance of Example~\ref{ex:lower-bound-attained}, so Theorem~\ref{thm:lower-bound} is still sharp with this restriction on $G$. When $m$ at least $2$, computational data suggests the bound in Theorem~\ref{thm:lower-bound} can be increased by $1$ plus a rounding error, but not more. We prove this for the case $m=2$ in the next section (Proposition~\ref{prop:Z-mod-p-lower}), and ask whether it holds for all $m\geq 2$ in Section~\ref{sec:other-questions}.

Meanwhile, for all abelian $G$, if $m$ is large enough (in particular if $m$ is greater than both $\log_3 |G|$ and the number of involutions in $G$), then the following ``hard floor" lower bound is better than Theorem~\ref{thm:lower-bound}:

\begin{proposition}\label{prop:gfield-at-least-3}
Let $G$ be a finite abelian group, and $V$ a nontrivial, non-modular, finite-dimensional representation of $G$. Then $\gfield(G,V)=2$ if and only if all the nontrivial characters in $V$ are involutions; otherwise, it is at least $3$.

In particular, if $m$ is the number of distinct nontrivial characters in $V$ and $\tau$ is the number of involutions in $G$, then the condition
\[
m > \tau
\]
implies that
\[
\gfield(G,V) \geq 3.
\]
\end{proposition}

\begin{proof}
The lattice $L(G,\Supp V)$ contains no points of degree $1$ because $\Supp V$ does not contain the trivial character. Meanwhile, the points $\bfa=(a_\chi)_{\chi \in \Supp V}$ of degree $2$  are either of the form $a_{\chi^\star} = 2$ and $a_\chi=0$ for $\chi \neq \chi^\star$, if $\chi^\star\in \Supp V$ is an involution, or  $a_{\chi^\star}=a_{(\chi^\star)^{-1}}=1$ and $a_\chi=0$ for $\chi \neq \chi^\star, (\chi^\star)^{-1}$, if $\chi^\star$ is not an involution. Thus they are in bijection with the disjoint union $\mathcal{I}\cup \mathcal{P}$ of the set $\mathcal{I}$ of involutions, and the set $\mathcal{P}$ of pairs of distinct inverses, contained in $\Supp V$. Counting elements, we have  $|\mathcal{I}|+2|\mathcal{P}| \leq m$. It follows that there are $m$ points of degree $2$ if and only if 
\[
m = |\mathcal{I}\cup \mathcal{P}| \leq |\mathcal{I}| + |\mathcal{P}| \leq m - |\mathcal{P}|,
\]
i.e., $|\mathcal{I}| = m$ and $|\mathcal{P}|=0$, i.e., $\mathcal{I} = \Supp V$, i.e., every single element of $\Supp V\subset \widehat G$ is an involution. Meanwhile, all the points of degree $2$ are linearly independent because they have pairwise disjoint support, so they generate a full-rank sublattice if and only if there are $m=\operatorname{rk} L(G,\Supp V)$ of them. This proves the first part of the proposition.

The second part follows because $G$ is isomorphic to $\widehat G$; thus if $m>\tau$, there are not enough involutions in $\widehat G$ to exhaust $\Supp V$.
\end{proof}

\begin{corollary}\label{cor:gfield-at-least-3}
If $G$ is finite abelian but not an elementary abelian $2$-group, and $V$ is a faithful non-modular finite-dimensional representation, then $\gfield(G,V)\geq 3$.
\end{corollary}

\begin{proof}
We prove the contrapositive. Proposition~\ref{prop:gfield-at-least-3} tells us that if $\gfield(G,V) < 3$, then either $G$ is trivial (and then it is an elementary abelian $2$-group), or else $\gfield(G,V)=2$ and all the nontrivial characters in $V$ are involutions. But then the image of $G$ in $GL(V)$, when written on the diagonal basis for $V$, lands inside the group of $\pm 1$ diagonal matrices, which is an elementary abelian $2$-group. Since $V$ is presumed faithful, we then have that $G$ is a subgroup of an elementary abelian $2$-group, so it is itself an elementary abelian $2$-group.
\end{proof}

In general, $\gfield(G,V)$ is  lower than $\bfield(G,V)$:

\begin{example}\label{ex:g<b}
By a computer calculation, $\gfield(\ZZ/17\ZZ, [\{8,10,11\}]) = 5$ while $\bfield(\ZZ/17\ZZ,[\{8,10,11\}])=6$.
\end{example}

However, we observed that in many of the small examples we computed, there was equality between $\gfield$ and $\bfield$; to illustrate, Example~\ref{ex:g<b} is, up to equivalence (as defined at the end of Section~\ref{sec:fields-and-lattices}, i.e., under automorphisms of $G$), the only example that occurs for $G=\ZZ/p\ZZ$ with $p\leq 17$ and $m\leq 3$. Motivated by this observation, we include some results that give conditions guaranteeing this equality. The first of these follows from the proofs of Theorems~\ref{thm:lower-bound} and \ref{thm:lower-bound-general}: when $\gfield(G,V)$ is close to the lower bounds given there, $\bfield(G,V)$ is no bigger. For other such results, see Propositions~\ref{prop:b=g-in-2d} and \ref{prop:extremal} below.

\begin{proposition}\label{prop:b=g-when-close-to-bottom}
If $G$ is a finite group and $V$ is a representation of $G$ of dimension $N$, and
\[
\gfield(G,V) < \sqrt[N]{2|G|},
\]
then
\[
\bfield(G,V) = \gfield(G,V).
\]
If $G$ is abelian and $V$ is non-modular, then the number of distinct nontrivial characters $m$ can take the place of $N$ in the hypothesis.
\end{proposition}

\begin{proof}
Consulting the above proof of Theorem~\ref{thm:lower-bound-general}, the first inequality in the chain of inequalities is 
\[
[\KK(V):\KK(V)^G] \leq [\KK(V):\KK(f_1,\dots,f_N)].
\]
Equality here implies that the minimum-degree transcendence basis $f_1,\dots,f_N$ already generates $\KK(V)^G$, in which case $\bfield(G,V)=\gfield(G,V)$. On the other hand, strict inequality is impossible under the given hypothesis: it would imply that 
\[
[\KK(V)^G : \KK(f_1,\dots,f_N)] \geq 2,
\]
so that 
\[
2|G|\leq [\KK(V):\KK(V)^G][\KK(V)^G:\KK(f_1,\dots,f_N)] \leq \deg(f_N)^N.
\]
But since $\deg(f_N) = \gfield(G,V)$, the hypothesis on $\gfield(G,V)$ rules this out. 

In the abelian, coprime characteristic case, we can either combine the conclusion just reached with Lemma~\ref{lem:distinct-nontrivial}, or else reason in parallel, following the proof of Theorem~\ref{thm:lower-bound}. A strict inequality $\gfield(G,V)<\bfield(G,V)$ would imply that, in the notation of the proof of Theorem~\ref{thm:lower-bound}, the points $\bfa_1,\dots,\bfa_m$ generate a lattice of index at least two in $L(G,\Supp V)$. But this would imply that
\[
2|G| =2 [\ZZ^m:L(G,\Supp V)] \leq \Vol(T) \leq \dots \leq \gfield(G,V)^m,
\]
and this is ruled out by the hypothesis.
\end{proof}

\subsection{Upper bound for $G=\ZZ/p\ZZ$ and related results}

We now develop the upper bound for the case $G=\ZZ/p\ZZ$. The argument bootstraps from information about the special case $m=2$ which is proven in Proposition~\ref{prop:(p+3)/2} in the next section. This information propagates to higher $m$ via the following proposition.

\begin{proposition}\label{prop:subsets}
Let $G=\ZZ/p\ZZ$. Let $S\subset \widehat G \setminus \{1\}$ be a set of distinct nontrivial characters. Let $S_1,S_2$ be nondisjoint subsets of $S$ with $S = S_1\cup S_2$. Then
\[
\bfield(G,S) \leq \max_{i\in \{1,2\}}(\bfield(G,S_i)).
\]
\end{proposition}

\begin{proof}
For any proper subset $S'\subset S$, the lattice $\ZZ^{S'}$ is naturally identified with a sublattice of $\ZZ^S$ along the embedding that maps a point of $\ZZ^{S'}$ to the point of $\ZZ^S$ with the same numbers in the $S'$-coordinates and zero in the $S\setminus S'$-coordinates. Then $L(G,S)$ is a full-rank sublattice of $\ZZ^S$, and $L(G,S_i) = L(G,S) \cap \ZZ^{S_i}$ for $i=1,2$. In what follows we make these identifications without further comment.

For $i=1,2$, let $\Gamma_i\subset L(G,S_i)$ be a generating set for $L(G,S_i)$ that realizes the bound $\bfield(G,S_i)$. We claim that $\Gamma_1\cup \Gamma_2$ is a generating set for $L(G,S)$, from which the proposition follows. Our work is reduced to establishing this claim.

Consider the natural projection
\[
\varphi: \ZZ^S\rightarrow \ZZ^{S\setminus S_1}
\]
obtained by forgetting the coordinates indexed by $S_1$. First note that the kernel of $\varphi$'s restriction to $L(G,S)$ is precisely $L(G,S_1)$. In particular, $\Gamma_1$ generates $\ker\varphi|_{L(G,S)}$. 

Next, we establish that $\varphi$'s restriction to $L(G,S_2)$ is surjective onto $\ZZ^{S\setminus S_1}$. The hypotheses on $S_1,S_2$ imply that $S_2$ is the disjoint union of $S\setminus S_1$ and $S_1\cap S_2$, and the latter is nonempty. Choose any 
\[
\chi^\star \in S_1\cap S_2.
\]
Note that $\chi^\star$ generates $\widehat G$, because $p$ is prime. Thus, for any integers $(a_\chi)_{\chi\in S\setminus S_1}$, the equation
\[
(\chi^\star)^a \prod_{\chi\in S\setminus S_1} \chi^{a_\chi} = 1\in \widehat G
\]
has an integer solution for $a$. Setting $a_{\chi^\star} :=a$, $a_{\chi'} = 0$ for $\chi'\in S_1\cap S_2 \setminus\{\chi^\star\}$ (if it is nonempty), and using the given numbers $a_\chi$ for $\chi\in S\setminus S_1$, we get a point $\bfa\in L(G,S_2)$ that maps under $\varphi$ to the point of $\ZZ^{S\setminus S_1}$ specified by $(a_\chi)_{\chi\in S\setminus S_1}$. Therefore, $\varphi$'s restriction to $L(G,S_2)$ is surjective onto $\ZZ^{S\setminus S_1}$, as claimed. Since $\Gamma_2$ generates $L(G,S_2)$, it follows that $\Gamma_2$'s image under $\varphi|_{L(G,S)}$ generates the entirety of $\ZZ^{S\setminus S_1}$.

It follows from Observation~\ref{obs:generate-over} (applied to the set $\Gamma_1\cup \Gamma_2$ and the map $\varphi|_{L(G,S)} : L(G,S)\rightarrow \ZZ^{S\setminus S_1}$) that $\Gamma_1\cup \Gamma_2$ generates $L(G,S)$, as desired.
\end{proof}

\begin{remark}
We observed that in the examples we computed with $G=\ZZ/p\ZZ$, when $S\subset S'$, it was extremely common that $\bfield(\ZZ/p\ZZ,S)\geq \bfield(\ZZ/p\ZZ,S')$. However, this was not guaranteed:

\begin{example}\label{ex:b-not-monotonic}
Let $G=\ZZ/41\ZZ$. Let $S=\{1,34\}$ (with characters represented by integers as in \eqref{eq:character-integer}) and let $S'=\{1,29,34\}$. Then $\bfield(G,S)=8$ while $\bfield(G,S')=9$.
\end{example}

\noindent Thus $\bfield(\ZZ/p\ZZ,S)$ is not a monotone nonincreasing function of $S$ (with respect to set containment order). Proposition~\ref{prop:subsets} can be interpreted as saying that it is ``not too far" from being a nonincreasing function of $S$.
\end{remark}

Modulo Proposition~\ref{prop:(p+3)/2}, proven below in the next section, we are ready to prove the upper bound on $\bfield(\ZZ/p\ZZ,V)$:

\begin{theorem}\label{thm:upper-bound}
Let $G= \ZZ/p\ZZ$. If $V$ is a non-modular representation of $G$, and there are $m\geq 3$ distinct, nontrivial characters occurring in $V$, then
\[
\bfield(G,V)\leq\frac{p+3}{2}.
\]
\end{theorem}

\begin{proof}
We proceed by induction on $m$. In the base case, there are $m=3$ distinct, nontrivial characters in $S := \Supp V$. Of the three possible pairs of these, at most one is a pair of inverses. Let $S_1,S_2$ be the other two pairs. By Proposition~\ref{prop:(p+3)/2} in the next section, $\bfield(G,S_i)\leq (p+3)/2$ for these two pairs. Since $S_1\cup S_2=S$ and $S_1\cap S_2$ is not empty, Proposition~\ref{prop:subsets} then tells us that $\bfield(G,S)\leq (p+3)/2$. This handles the base case.

For $m>3$, we again set $S:=\Supp V$; this time we take $S_1,S_2$ to be any two distinct ($m-1$)-subsets of $S$. Then they are again nondisjoint with union $S$, and for $i=1,2$ we have $\bfield(G,S_i)\leq (p+3)/2$ by the induction hypothesis. So we again conclude 
\[
\bfield(G,S)\leq \frac{p+3}{2}
\]
by Proposition~\ref{prop:subsets}.
\end{proof}

As mentioned at the beginning of the section, when $G = \ZZ/p\ZZ$, then $\beta(G,V) = p$ for any faithful representation $V$, and the same holds for $\bsep(G,V)$ if $\KK$ is algebraically closed (or even contains $p$th roots of unity), by \cite[Theorem~2.1]{domokos}; thus Theorem~\ref{thm:upper-bound} establishes a gap in this case between $\bfield$ and $\beta,\bsep$ when $m\geq 3$. Toward the signal processing application discussed in the introduction, we now verify (Proposition~\ref{prop:separating-over-R} below) that $\bsep(G,V)=p$ always when $\KK=\RR$ as well.

\begin{remark}
The argument for Proposition~\ref{prop:separating-over-R} adapts some ideas of \cite[Section~2]{domokos} to the $\KK=\RR$ setting. Lemma~\ref{lem:multigraded-separating}, in particular, generalizes \cite[Lemma~2.5]{domokos}.
\end{remark}

\begin{lemma}\label{lem:multigraded-separating}
Let $G$ be a finite group, let $V$ be a finite-dimensional $G$-representation over a field $\KK$, and let 
\[
V = \bigoplus_{i=1}^r V_i
\]
be a direct-sum decomposition of $V$ into (not necessarily irreducible) $G$-subrepresentations. Let $K[V]$ be $\NN^r$-graded by this decomposition, with any $f_i\in V_i^*\subset \KK[V_i]\subset \KK[V]$ assigned degree $\bfe_i\in \NN^r$. This induces an $\NN^r$-grading on $\KK[V]^G$. Let $U\subset \KK[V]^G$ be a separating set for the action of $G$ on $V$ that is also a $\KK$-linear subspace which is graded with respect to this $\NN^r$-grading. Then
\[
U\cap \KK[V_i]^G
\]
is a separating set for the action of $G$ on $V_i$, for each $i=1,\dots,r$.
\end{lemma}

\begin{proof}
The action of $G$ respects the $\NN^r$-grading of $\KK[V]$ because the $V_i\subset V$ are subrepresentations; it follows that $\KK[V]^G$ inherits the $\NN^r$-grading from $\KK[V]$. Because $U$ is $\NN^r$-graded, it has a basis $B$ consisting of forms that are multihomogeneous with respect to this $\NN^r$-grading, and this basis must, like $U$, form a separating set for $\KK[V]^G$. 

Fix any $V_i$, and consider any two distinct orbits $\mathcal{O}_1,\mathcal{O}_2$ of $G$ contained in $V_i$. For $f\in B$, we have
\[
\deg f = \sum_{j=1}^r c_j\bfe_j,
\]
with the $c_j\in \NN$. If for some $j\neq i$ we have $c_j\neq 0$, then $f$ is homogeneous of positive degree in the coordinate functions on $V_j$, which vanish identically on $V_i$, so then $f$ vanishes identically on $V_i$. In particular, in that case $f$ fails to separate $\mathcal{O}_1$ from $\mathcal{O}_2$. 

Because $G$ is finite, the separating set $B$ must separate all orbits. In particular, there must be some $f\in B$ that can separate $\mathcal{O}_1$ from $\mathcal{O}_2$; it follows from the previous paragraph that $\deg f = c_i\bfe_i$. This is equivalent to the statement that $f\in \KK[V_i]$. Because $f\in U$ is an invariant, in fact we have $f\in \KK[V_i]^G$.

Thus $B\cap \KK[V_i]^G$ separates any two orbits of $G$ on $V_i$; it follows that its linear span $U\cap \KK[V_i]^G$ does as well.
\end{proof}

\begin{proposition}\label{prop:separating-over-R}
If $p$ is a prime number and $V$ is a faithful, finite-dimensional representation of $G = \ZZ/p\ZZ$ over the field $\RR$ of real numbers, then
\[
\bsep(G,V) = p.
\]
\end{proposition}

\begin{proof}
Set $V_\CC=\CC\otimes_\RR V$ as in Lemma~\ref{lem:base-change}, and embed $V$ in $V_\CC$ in the natural way (as $1\otimes_\RR V$). Because the invariants over $\CC$ are $\CC$-linearly spanned by the invariants over $\RR$, and because there are fewer orbits to distinguish on $V$ than on $V_\CC\supset V$, we have $\bsep(G,V)\leq \bsep(G,V_\CC)$.   Also, $\bsep(G,V_\CC)=p$ by \cite[Theorem~2.1]{domokos}. So what needs to be shown is that $\bsep(G,V)$ is not lower than $p$, i.e., that $\RR[V]^G_{\leq d}$ cannot be a separating set if $d<p$.

Let
\[
V = \bigoplus V_i
\]
be a decomposition into irreducible subrepresentations over $\RR$.

{\em Case 1: All $V_i$ are one-dimensional.} This implies $p=2$. Faithfulness of $V$ then implies there is a nontrivial $V_i$, with $G$ acting by the sign representation; Lemma~\ref{lem:multigraded-separating} with $U=\RR[V]^G_{\leq d}$ implies that for $\RR[V]^G_{\leq d}$ to be separating (for $G$ on $V$), $\RR[V_i]^G_{\leq d} = \RR[V]^G_{\leq d}\cap \RR[V_i]^G$ must be separating for $G$ on $V_i$. The sign representation has no degree $1$ invariants, so $d$ must be at least $2$ ($=p$) for this to hold.

{\em Case 2: There is a $V_i$ of dimension $>1$.} Then because $G$ is abelian, $\dim_\RR V_i=2$ and $(V_i)_\CC := \CC\otimes_\RR V_i$ decomposes into a pair of (nontrivial) inverse characters of $G$. If $x_1,x_2$ are dual to the diagonal basis for $(V_i)_\CC$, then one can plot the lattice $L(G,(V_i)_\CC)$ to see that the only invariants of degree $<p$ in $\CC[(V_i)_\CC]^G$ are generated by $x_1x_2$.\footnote{Alternatively, consult the proof of Proposition~\ref{prop:extremal} below in the case $m=2$, after automorphing $G$ so the characters are $\pm 1$ (where characters are represented by integers as in \eqref{eq:character-integer}), to reach the same conclusion.} Since $\CC[(V_i)_\CC]^G$ is $\CC$-spanned by $\RR[V_i]^G$, we conclude that the only invariants of degree $<p$ in $\RR[V_i]^G$ are generated by the unique (up to $\RR^\times$-scaling) real invariant in the $\CC$-span of $x_1x_2$, which is the squared $2$-norm with respect to a $G$-invariant inner product on $V_i$. The $2$-norm cannot separate distinct $G$-orbits in $V_i$ that lie in the same origin-centered circle (these exist because $G$ is finite). Since $\RR[V_i]^G_{\leq d} = \RR[V]^G_{\leq d} \cap \RR[V_i]^G$, it follows from Lemma~\ref{lem:multigraded-separating} (with $U=\RR[V]^G_{\leq d}$) that  if $d<p$, then $\RR[V]^G_{\leq d}$ is not separating for the action of $G$ on $V$.
\end{proof}

Before concluding the section, we note that, although it was not needed in the proof of Theorem~\ref{thm:upper-bound}, a similar statement to Proposition~\ref{prop:subsets} holds for $\gfield$. We can drop the restriction to $G=\ZZ/p\ZZ$ and the hypothesis that the sets $S_1,S_2$ are nondisjoint, and the proof is much quicker.

\begin{proposition}\label{prop:gfield-subsets}
Let $G$ be a finite abelian group. Let $S\subset \widehat G \setminus \{1\}$ be a set of distinct nontrivial characters. Let $S_1,S_2$ be subsets of $S$ with $S = S_1\cup S_2$. Then
\[
\gfield(G,S) \leq \max_{i\in \{1,2\}}(\gfield(G,S_i)).
\]
\end{proposition}

\begin{proof}
We follow the notation and conventions of the proof of Proposition~\ref{prop:subsets}, in particular regarding $\ZZ^{S_1},\ZZ^{S_2}$ as sublattices of $\ZZ^S$ via the natural embeddings. For $i=1,2$, let $\Gamma_i'\subset L(G,S_i)$ be a generating set for a full-rank sublattice of $L(G,S_i)$ that realizes the bound $\gfield(G,S_i)$. Because $S_1\cup S_2 = S$, the group homomorphism
\begin{align*}
    \ZZ^{S_1}\times \ZZ^{S_2}&\rightarrow \ZZ^{S}\\
    (\bfa_1,\bfa_2)&\mapsto \bfa_1+\bfa_2
\end{align*}
is surjective (as the image contains every standard basis vector). Then the composition of this map with the canonical homomorphism $\ZZ^{S} \rightarrow \ZZ^S / \langle \Gamma_1',\Gamma_2'\rangle$ is surjective. But it factors through $\ZZ^{S_1}/\langle \Gamma_1'\rangle \times \ZZ^{S_2}/\langle \Gamma_2'\rangle$, which by the choice of $\Gamma_1',\Gamma_2'$ is a finite group. Thus $\ZZ^S / \langle \Gamma_1',\Gamma_2'\rangle$ is finite. In other words, $\Gamma_1'\cup \Gamma_2'$ generates a full-rank sublattice of $\ZZ^S$; and it realizes the bound in the proposition.
\end{proof}

\section{Two distinct nontrivial  characters}\label{sec:two-characters}

In this section we obtain detailed information on $\bfield(G,V)$ and $\gfield(G,V)$ in the special situation that the number of distinct, nontrivial characters of (abelian) $G$ appearing in $V$ is exactly two. Proposition~\ref{prop:b=g-in-2d} shows that $\bfield=\gfield$ always in this situation. The rest of the section restricts attention to the case $G=\ZZ/p\ZZ$ for $p$ an odd prime.  Proposition~\ref{prop:q+b+r-1} gives an upper bound on $\bfield$ that becomes an exact formula when the ratio between the two characters can be expressed as an integer that is small in comparison with $p$ or almost divides it. Proposition~\ref{prop:(p+3)/2} deduces from this a global upper bound as long as the two characters are not inverses; it is a key lemma for Theorem~\ref{thm:upper-bound}. Proposition~\ref{prop:hilbert-series} gives some information about the form of the Hilbert series of the ring $\KK[V]^G$ when $V$ is free of repeated or trivial characters. Proposition~\ref{prop:Z-mod-p-lower} mildly improves the lower bound of Theorem~\ref{thm:lower-bound} when $G=\ZZ/p\ZZ$, and characterizes the representations that attain the improved lower bound.

As in the previous section, the proofs use freely, and usually without explict comment, the results of Section~\ref{sec:fields-and-lattices}, in particular Lemmas~\ref{lem:equivalence2} and \ref{lem:distinct-nontrivial}.

\begin{proposition}\label{prop:b=g-in-2d}
If $G$ is a finite abelian group and $V$ is a finite-dimensional, faithful, non-modular representation such that the number $m$ of distinct, nontrivial characters of $G$ appearing in $V$ is exactly two, then 
\[
\gfield(G,V) = \bfield(G,V).
\]
\end{proposition}

\begin{proof}
We can replace $V$ in $\gfield(G,V)$ and $\bfield(G,V)$ by $\Supp V$; thus it suffices to show for a rank-2 lattice $L(G,V)\subset \ZZ^2$ that its generation degree is not bigger than its full-rank degree. We establish this statement in the following form: {\em if there are two linearly independent elements of $L(G,V)$ inside $d\Delta_2$ for some real number $d$, then there is a basis for $L(G,V)$ inside $d\Delta_2$.}

Assume that there exist two linearly independent elements of $L(G,V)$ inside $d\Delta_2$. Choose a pair $\bfa_1,\bfa_2$ of such elements, subject to the requirement that the area of the closed triangle $T$ with vertices $0,\bfa_1,\bfa_2$ is minimal among such pairs. (Since only finitely many points of $L(G,V)$ lie in $d\Delta_2$, this is possible.) We claim that $\bfa_1,\bfa_2$ form a basis of $L(G,V)$. If not, then by \cite[Chapter~1, Section~3, Theorem~4]{lekkerkerker-gruber} (``theorem on lattice triangles"), there is a point $\bfc$ of $L(G,V)$ in $T$ other than $0,\bfa_1,\bfa_2$. Since $0,\bfa_1,\bfa_2\in d\Delta_2$, and $d\Delta_2$ is convex, it follows that $T\subset d\Delta_2$, thus $\bfc\in d\Delta_2$ as well. As $\bfa_1,\bfa_2$ are linearly independent, $\bfc$ is linearly independent with at least one of them, say $\bfa_1$. Then the closed triangle with vertices $0,\bfa_1,\bfc$ is properly contained in $T$ and so has smaller area, contradicting the minimality of $T$. This proves the claim.
\end{proof}

For the rest of the section, we restrict our attention to the situation that $G=\ZZ/p\ZZ$, with $p$ an odd prime. The following gives an exact formula for $\bfield(G,V)$ when the two characters in $V$ are related by multiplication by an integer that is either small in comparison with $p$ or almost divides it.

\begin{proposition}\label{prop:q+b+r-1}
Let $G=\ZZ/p\ZZ$ for a prime number $p\geq 3$. Let $V$ be a finite-dimensional, non-modular representation of $G$ such that the number $m$ of distinct, nontrivial characters appearing in $V$ is exactly two. Represent these characters by integers $A_1,A_2$ as in equation~\eqref{eq:character-integer}. Let $b$ be the positive integer less than $p$ satisfying $bA_1 = A_2\pmod p$. Write
\[
p = qb + r
\]
with $0 < r < b$. Then
\begin{equation}\label{eq:q+b+r-1}
\bfield(G,V) \leq q + b + r-1,
\end{equation}
with equality if either
\[
q \geq r(r-1)
\]
or
\[
b\gg r\text{ and }q>r.
\]
In particular, equality holds in \eqref{eq:q+b+r-1} for $q$
sufficiently high depending only on $r$. 
\end{proposition}

\begin{remark}
A more general sufficient condition for equality in \eqref{eq:q+b+r-1}, implied by both of the sufficient conditions in the proposition statement, is constructed in the course of the proof; see \eqref{eq:suff-high-q} below. 
Also, as $\bfield(G,V)$ is symmetric with respect to $A_1,A_2$, the same statement holds with $0<b'<p$ such that $A_1=b'A_2\pmod{p}$ in place of $b$, and $q'>0$, $0 < r'<p$ such that $p=q'b'+r'$ in place of $q,r$.
\end{remark}

\begin{proof}[Proof of Proposition~\ref{prop:q+b+r-1}]
Note that $b\geq 2$ because $A_1\neq A_2$. The definition of $b$ lets us normalize equation \eqref{eq:lattice-eq} defining $L(G,\Supp V)$ to
\[
a_1 + ba_2 = 0\pmod p.
\]
By substitution, it contains the points $(r,q)$ and $(r+b,q-1)$. By computing a determinant, these form a basis for $L(G,\Supp V)$ in view of Lemma~\ref{lem:index}. Since (in view of $b\geq 2$) the higher-degree of the two points is $(r+b,q-1)$, this proves that the generation degree is at most $r+b+q-1$. The following argument shows that, for sufficiently high $q$ as in the proposition, there does not exist another point of lower degree than this that could be part of a basis for $L(G,\Supp V)$.

We need not consider points with $a_1\geq p$ or $a_2\geq p$, because $q+b+r-1$ is already $\leq p$ (with equality if and only if $q=1$). So we assume $a_1,a_2<p$ for any point that is in contention with $(r,q)$ and $(r+b,q-1)$ as a part of a minimum-degree basis.

No nonzero first-quadrant point of $L(G,\Supp V)$ lying below the line $a_2=q$ is lower-degree than $r+b+q-1$, as follows. Every first-quadrant point of $L(G,\Supp V)$ with $a_2\leq q$ and $a_1 < p$ is on the line $a_1 + ba_2 = p$, by definition of $q$. And $(r,q)$ and $(r+b,q-1)$ are already the lowest-degree points in the first quadrant on that line, because $b\geq 2$ so that the degrees of points on this line increase as $a_2$ decreases.

We show below that for sufficiently high $q$ as in the proposition, all nonzero first-quadrant points of $L(G,\Supp V)$ to the left of the line $a_1=r$ are also of degree at least $r+b+q-1$. In this paragraph we argue that this will complete the proof. If $\bfa$ is any non-multiple of $(r,q)$ that is not below $a_2=q$ or to the left of $a_1=r$, then there is a lower-degree, nonzero first-quadrant point $\bfb$ obtained from $\bfa$ by subtracting a nontrivial multiple of $(r,q)$, that {\em is} either below $a_2=q$ or to the left of $a_1=r$. For the promised sufficiently high $q$, $\bfb$ must have degree at least $r+b+q-1$; therefore so must $\bfa$. Thus, this sufficiently high $q$ will guarantee that {\em all} first-quadrant points of $L(G,\Supp V)$ except for the multiples of $(r,q)$ have degree at least $r+b+q-1$. This will imply that $\bfield(G,V)$ is at least (and thus equal to) $r+b+q-1$, completing the proof. It remains to exhibit the sufficiently high $q$ with the promised property.

Let $C$ be the set of nonzero points of $L(G,\Supp V)$ satisfying $0 \leq a_1<r$ and $0\leq a_2<p$. Note that no point of $C$ satisfies $a_1=0$, because $p$ is prime and the origin is excluded from $C$ by construction. Thus if $r=1$, then $C$ is empty and there are no competing points, so equality is attained in \eqref{eq:q+b+r-1}. So we may suppose that $r>1$. Every point of $C$ satisfies $a_2>q$, because every first-quadrant point of $L(G,\Supp V)$ that is on or below $a_2=q$ and satisfies $a_1<p$ is already on the line $a_1+ba_2=p$, as discussed above, and any point on this line has $a_1\geq r$ by definition of $r$ so is not in $C$.

Let $\bfa = (1,Y)$ be the unique point of $C$ with $a_1=1$. (Existence and uniqueness follow from the primality of $p$.) Then $(r,Yr)$ is in $L(G,\Supp V)$ as well, and therefore so is $(r,Yr) - (r,q) = (0, Yr-q)$. It follows that $Yr-q$ is a multiple of $p$ (again by the latter's primality). Let $Yr-q = Kp$, where $K$ is an integer. So $Y= (Kp+q)/r$.

Again by the primality of $p$, there is exactly one point $\bfa_j$ of $C$ with $a_1=j$ for each $j=1,\dots,q-1$; this point has the form
\[
\bfa_j = j(1,Y) - \ell(0,p) = (j,jY-\ell p)
\]
for some nonnegative integer $\ell$. (To avoid clutter, we suppress from the notation the dependence of $\ell$ on $j$.) The points $\bfa_j$, $j=1,\dots,r-1$ exhaust $C$. Substituting the above expression for $Y$, we get
\[
\bfa_j = \left(j, j\frac{Kp + q}{r} - \ell p\right) = \left(j,\frac{(jK - \ell r)p + jq}{r}\right).
\]
Considering that all points of $C$ are above the line $a_2=q$ as discussed above, the integer $jK-\ell r$ must be at least $1$: otherwise, the $a_2$-coordinate of $\bfa_j$ would be less than $q$, since $j$ is less than $r$. Because $j\geq 1$ as well, it follows that
\[
\deg \bfa_j = j + \frac{(jK - \ell r)p + jq}{r} \geq 1 + \frac{p + q}{r}
\]
for all the points $\bfa_j$ of $C$.

Solving the inequality 
\[
1 + \frac{p + q}{r}\geq  r+b+q-1
\]
for $q$ after substituting $p=qb+r$, we get
\begin{equation}\label{eq:suff-high-q}
q\geq \frac{r^2+rb-3r}{b-r+1}.
\end{equation}
This is the sufficiently high $q$: when this inequality holds, all points of $C$, and thus (as discussed above) all points of $L(G,\Supp V)$ in the first quadrant other than multiples of $(r,q)$, have degree at least $r+b+q-1$. Thus, for such $q$, we have equality in \eqref{eq:q+b+r-1}.

The substitution $\widehat b = b-r+1$, $b=\widehat b + r - 1$ clarifies that the right side of \eqref{eq:suff-high-q} is a decreasing function of $b$: it becomes
\begin{equation}\label{eq:b-hat}
r + \frac{2r(r-2)}{\widehat b}.
\end{equation}
The definition of $r$ guarantees that $\widehat b = b - r + 1 \geq 2$; thus \eqref{eq:b-hat} is at most $r + r(r-2) = r(r-1)$. It follows that the condition 
\[
q\geq r(r-1)
\]
guarantees  \eqref{eq:suff-high-q} and thus equality in \eqref{eq:q+b+r-1}. Note that this condition depends only on $r$ and not $b$.

Meanwhile, if $b$ is large next to $r$ (in particular, if $b >2r(r-2)+r-1$), then the fraction in \eqref{eq:b-hat} is less than one, so $q>r$ guarantees \eqref{eq:suff-high-q} and thus equality in \eqref{eq:q+b+r-1}. This completes the proof.
\end{proof}

\begin{example}\label{ex:r=1-and-2}
The $r=1$ and $r=2$ cases of Proposition~\ref{prop:q+b+r-1} say, respectively (in the notation of the proposition) that if $p=1\pmod b$, then
\[
\bfield(\ZZ/p\ZZ,V) =  \frac{p-1}{b}+b,
\]
while if $p=2\pmod b$ and $b$ is anything other than $p-2$, then
\[
\bfield(\ZZ/p\ZZ,V) = \frac{p-2}{b} + b + 1.
\]
\end{example}

The following is a sharp upper bound on $\bfield(\ZZ/p\ZZ,V)$ when the two nontrivial characters in $V$ are not inverse to each other.

\begin{proposition}\label{prop:(p+3)/2}
Let $G=\ZZ/p\ZZ$ for a prime number $p$. 
Let $V$ be a faithful, finite-dimensional, non-modular representation of $G$ such that the number $m$ of distinct, nontrivial characters of $G$ in $V$ is exactly two. Then unless these characters are inverses of each other, we have
\[
\bfield(G,V) \leq \frac{p+3}{2}.
\]
\end{proposition}

\begin{proof}
Representing characters by integers as in \eqref{eq:character-integer}, the lattice $L(G,\Supp V)$ has the form 
\begin{equation}\label{eq:lattice-As}
A_1a_1+A_2a_2= 0 \pmod p
\end{equation}
where $A_1$ and $A_2$ are neither equal nor inverse mod $p$. As in Proposition~\ref{prop:q+b+r-1} and the remark following it, write $b,b'$ for the unique positive integers less than $p$ satisfying $bA_1=A_2$ and $A_1=b'A_2$; then use division with remainder to find positive integers $q,q',r,r'$ satisfying
\[
p = qb+r = q'b'+r'
\]
with $0<r<b$ and $0<r'<b'$. The argument breaks into cases depending on $q$ and $q'$.

{\em Case 1: $q=1$ or $q'=1$.} We handle the case $q'=1$; the argument is the same for $q=1$ except with the axes reversed. 

Equation \eqref{eq:lattice-As} can be rewritten
\[
b'a_1 + a_2= 0\pmod p.
\]
As in the proof of Proposition~\ref{prop:q+b+r-1}, let $Y$ be the unique integer with $0<Y<p$ so that the point $\bfa = (1,Y)$ solves this equation for $(a_1,a_2)$, i.e., belongs to $L(G,\Supp V)$. Because $q'=1$, we have $b' \geq (p+1)/2$; thus $Y\leq (p-1)/2$. Also, $Y\geq 2$, because $Y=1$ would imply that $b'=-1\pmod p$, and thus that $A_1,A_2$ are inverse mod $p$, contrary to hypothesis. So $2\leq Y \leq (p-1)/2$.

Now let $j$ be the smallest natural number so that $jY$ exceeds $p$. Then $jY \leq p + Y - 1$, so
\[
j \leq \frac{p + Y - 1}{Y}.
\]
The point $(j,jY-p)$ lies in the first quadrant, and a determinant calculation shows (in view of Lemma~\ref{lem:index} and the equality of index with the area of a fundamental parallelogram) that it and $(1,Y)$ generate $L(G,\Supp V)$. We have
\[
\deg (j,jY-p) \leq \frac{p + Y - 1}{Y} + Y-1 = Y + \frac{p-1}{Y}.
\]
This is a convex function of $Y$ (for $Y>0$), and it is equal to $(p+3)/2$ at $Y=2$ and $Y=(p-1)/2$. It follows that \[
\deg(j,jY-p) \leq (p+3)/2
\]
on the full interval $2\leq Y \leq (p-1)/2$. As $\deg (1,Y) \leq 1 + (p-1)/2 < (p+3)/2$ on this interval as well, this completes the proof in Case 1.

{\em Case 2: $q\geq 3$ or $q'\geq 3$.} We check that when $q\geq 3$, the bound \eqref{eq:q+b+r-1} given by Proposition~\ref{prop:q+b+r-1} is less than or equal to $(p+3)/2$; the argument is the same for $q'\geq 3$ except with the axes reversed.

Substituting $p-qb$ for $r$ in \eqref{eq:q+b+r-1} yields 
\[
\bfield(G,V) \leq p+q+(1-q)b-1.
\]
In view of the assumption $q\geq 3$, the right side is a decreasing function of $b$ for fixed $p$ and $q$. Meanwhile, the definition of $q,r$ implies that $(q+1)b=p+b-r\geq p+1$, so $b \geq (p+1)/(q+1)$. Therefore, 
\begin{align*}
\bfield(G,V) &\leq p+q+(1-q)b-1\\
&\leq p+q+(1-q)\frac{p+1}{q+1} - 1 \\
&= q + 2\frac{p+1}{q+1}-2.
\end{align*}
The right side is a convex function of $q$ (for $q>-1$), and it evaluates to $(p+3)/2$ at $q=3$ and $q=(p-1)/2$. We have $q\geq 3$ by assumption, and since $b\geq 2$ (as $A_1\neq A_2\mod p$) and $p$ is prime, we also have $q\leq (p-1)/2$ in view of the definition of $q$. So we can conclude that $\bfield(G,V)\leq (p+3)/2$, completing the proof in Case 2.

{\em Case 3: $q=q'=2$.} In this case, $L(G,\Supp V)$ contains the points  $\bfa = (r,2)$ and $\bfb = (2,r')$. If they are linearly dependent, a determinant calculation gives $rr'=4$, so at least one of $r,r'$ is even; but this is a contradiction because then either $p=qb+r$ or $p=q'b'+r'$ would imply $p$ is even. So $\bfa,\bfb$ are linearly independent.

Meanwhile, $q=2$ implies that 
\[
b\geq (p+1)/3,
\]
thus that
\[
r = p-2b \leq \frac{p-2}{3}.
\]
So
\[
\deg \bfa = 2+r \leq \frac{p+4}{3}.
\]
The same bound is satisfied by $\bfb$, by the same argument with axes reversed. Since $\bfa,\bfb$ are linearly independent, we obtain $\gfield(G,V) \leq (p+4)/3$, 
and thus 
\[
\bfield(G,V)\leq \frac{p+4}{3}
\]
by Proposition~\ref{prop:b=g-in-2d}. Since $(p+4)/3 < (p+3)/2$, this completes the proof of Case 3.
\end{proof}

\begin{remark}
Proposition~\ref{prop:(p+3)/2} is sharp. It is attained by the equivalence classes of $S=\{1,2\}$ and $S'=\{1,(p-1)/2\}$ (with integers representing characters as in \eqref{eq:character-integer}). This follows from Example~\ref{ex:r=1-and-2} since $p=1\pmod b$ in these cases. In the excluded case that $A_1,A_2$ are inverses, $\bfield(G,V)=p$ (as can be seen either from Example~\ref{ex:r=1-and-2} since $b=p-1$ so again $p=1\pmod b$, or from Proposition~\ref{prop:extremal} below in the case $m=2$).
\end{remark}

The following result yields information about the Hilbert series of the invariant ring $\KK[V]^G$ in the case that $V$ has no trivial or repeated characters. It lies somewhat to the side of our main line of inquiry but is interesting in its own right, and part of it is used to prove Proposition~\ref{prop:Z-mod-p-lower} below.

\begin{proposition}\label{prop:hilbert-series}
Let $G=\ZZ/p\ZZ$ with $p$ an odd prime. Let $V$ be a non-modular representation of $G$ such that the number $m$ of distinct, nontrivial characters of $G$ in $V$ is exactly two. Then the degrees of the points of $L(G,\Supp V)$ contained in
\[
T := \{(a_1,a_2) :0\leq a_1,a_2 < p\}
\]
are all distinct, and the set $D$ of nonzero degrees among them is contained in $\{2,3,\dots,2p-2\}$ and satisfies the following three properties:
\begin{enumerate}
    \item $D$ is stable under the substitution $d\mapsto 2p-d$.\label{pro:sym}
    \item There is exactly one element of $D$ in each nonzero residue class mod $p$, and $p\notin D$. \label{pro:one-in-each}
    \item Fix any $d\in D$. Then there is no element of $D$ congruent to $p\pmod{d}$.\label{pro:avoids-p-mod-d}
\end{enumerate}

In particular, if furthermore $N=\dim_\KK V = 2$, then the Hilbert series of the ring $\KK[V]^G$ has the form
\[
H(\KK[V]^G,t) = \frac{1 + \sum_{d\in D}t^d}{(1-t^p)^2},
\]
with $D$ as above.
\end{proposition}

\begin{remark}
The tricky part of the proposition is property~\ref{pro:avoids-p-mod-d}, the rest is straightforward. Note that property~\ref{pro:one-in-each} together with $D\subset\{2,3,\dots,2p-2\}$ imply that $p-1,p+1\in D$.
\end{remark}

\begin{example}
Take $G=\ZZ/13\ZZ$, and let $V$ be the representation defined by the characters $1$ and $3$ (where characters are represented by integers as in \eqref{eq:character-integer}). Then the Hilbert series of $\KK[V]^G$ is
\[
\frac{1+t^5+t^7+t^9+t^{10}+t^{11}+t^{12}+t^{14}+t^{15}+t^{16}+t^{17}+t^{19}+t^{21}}{(1-t^{13})^2}.
\]
Observe that the set $D$ of nonzero exponents in the numerator is symmetric with respect to $d\mapsto 26-d$, hits each of the 12 nonzero residue classes mod $13$ exactly once, and for each exponent $d$, there is no exponent that is $13$ mod $d$---for example, no exponent is $3$ mod $5$, $6$ mod $7$, $4$ mod $9$, or $2$ mod $11$. (It happens that all other residue classes mod $5$, mod $7$, mod $9$, and mod $11$ do occur. One can show in general that the elements $d$ of $D$ such that $D$ contains every residue class mod $d$ except $p$'s are exactly those that are equal to the degree of a minimal generator of the Hilbert ideal of $V$. We omit the proof.)
\end{example}

\begin{proof}[Proof of Proposition~\ref{prop:hilbert-series}]
We argue both parts of the proposition together in the situation that $N= 2$. The first part follows for $N>2$ (but $m=2$) by replacing $L(G,V)$ with $L(G,\Supp V)$. 

The Hilbert series is not affected by base change or by the choice of indeterminates, so we assume $\KK$ contains $p$th roots of unity and we have chosen $x_1,x_2\in V^*$ on which $G$ acts by distinct, nontrivial characters. Thus the $\exp$ map $\bfa\mapsto \bfx^\bfa$ carries the semigroup $L(G,V)\cap \NN^2$ to the $\KK$-basis for $\KK[V]^G$ consisting of invariant monomials in $x_1,x_2$.

The lattice $L(G,V)$ contains $p\bfe_1$ and $p\bfe_2$ (where as usual $\bfe_1,\bfe_2$ are standard unit basis vectors), and they generate the sublattice $p\ZZ^2\subset L(G,V)$. The $\exp$ map converts them into the homogeneous system of parameters $x_1^p, x_2^p$ for $\KK[V]^G$, and  converts the set of elements of $L(G,V)$ lying in $T$ into a module basis for $\KK[V]^G$ over the parameter subring $\KK[x_1^p,x_2^p]$. (One can see this by suitably specializing \cite[Theorem~3.1]{huffman}, although it is probably easier to deduce it from the fact that $T$ is a fundamental domain for $p\ZZ^2$ that tiles the first quadrant.) Thus, the Hilbert series has the form
\[
H(\KK[V]^G,t) = \frac{\sum_{\bfa\in T\cap L(G,V)} t^{\deg \bfa}}{(1-t^p)^2},
\]
and the problem is to show that the numerator is $1+\sum_{d\in D} t^d$ with $D$ as described in the proposition.

We denote the characters of $V^*$ 
by integers $A_1,A_2$ as in \eqref{eq:character-integer}. They are distinct mod $p$ by assumption. Thus, for each residue class $r$ mod $p$, there exists a unique solution mod $p$ to the system of equations
\begin{align*}
    A_1a_1 + A_2a_2 &= 0\pmod{p}\\
    a_1 + a_2 &= r\pmod{p},
\end{align*}
and thus a unique integer solution $(a_1,a_2)\in T$. The system characterizes lattice points with degree equal to $r$ mod $p$; thus uniqueness of the solution inside $T$ implies that all the points of $T\cap L(G,V)$ have distinct degrees (in fact, even distinct mod $p$). Also, $0$ occurs as a degree because $(0,0)\in T\cap L(G,V)$. Thus
\[
\sum_{\bfa \in T\cap L(G,V)} t^{\deg \bfa} = 1 + \sum_{d\in D}t^d,
\]
where $D$ is a set of positive integers. Uniqueness and existence of the system's solution within $T$ then imply that $D$ has property~\ref{pro:one-in-each} of the proposition. (The assertion $p\notin D$ is part of this conclusion; the solution with $r=0\pmod{p}$ is already accounted for by $(0,0)$.)

Because $p$ is prime and $A_1,A_2$ are nontrivial characters (i.e., nonzero mod $p$), $T\cap L(G,V)$ has no points along the coordinate axes except for $(0,0)$, thus all points of $T\cap L(G,V)$ except for $(0,0)$ lie in the interior $\operatorname{int} T$ of $T$; thus $D$ is the set of degrees of points of $\operatorname{int} T \cap L(G,V)$. Both $L(G,V)$ and $\operatorname{int} T$ are setwise stable under the map $\bfa \mapsto p(\bfe_1+\bfe_2)-\bfa$. Because $\deg p(\bfe_1+\bfe_2)= 2p$ and $\deg:\NN^2 \rightarrow\NN$ is an additive map, this implies that $D$ has property~\ref{pro:sym}. 

The region $T$ contains no points of degree higher than $2p-2$; in view of property~\ref{pro:sym} this implies $D$ contains no points of degree lower than $2$, thus $D\subset\{2,3,\dots,2p-2\}$ as claimed.  It remains to prove property~\ref{pro:avoids-p-mod-d}.

Let $d\in D$ be arbitrary, and find the  $\bfa=(a_1,a_2) \in \operatorname{int} T \cap L(G,V)$ with $\deg \bfa = d$. The index of $p\ZZ^2$ in $\ZZ^2$ is $p^2$; in view of Lemma~\ref{lem:index} it follows that the group $L(G,V)/p\ZZ^2$ is of order $p$. Because $\bfa$ represents a nontrivial element of this group, it is therefore a generator. Thus, the $p-1$ points
\[
\bfa, 2\bfa, \dots, (p-1)\bfa \in \ZZ^2
\]
represent the $p-1$ nontrivial cosets of $L(G,V)/p\ZZ^2$. Since $T$ is a fundamental parallelotope for the lattice $p\ZZ^2$, the nonzero points of $T\cap L(G,V)$, which we have determined above are the same as the nonzero points of $\operatorname{int} T\cap L(G,V)$,  also represent the $p-1$ nontrivial cosets of $L(G,V)/p\ZZ^2$. It follows that each $j\bfa$ ($j=1,\dots,p-1$) is congruent mod $p$ to a distinct one of the elements of $\operatorname{int} T\cap L(G,V)$. By the definition of $T$ we then conclude that the points of $\operatorname{int}T\cap L(G,V)$ have the form
\[
j\bfa - \left\lfloor \frac{ja_1}{p}\right\rfloor p\bfe_1 - \left\lfloor \frac{ja_2}{p}\right\rfloor p\bfe_2
\]
for $j=1,\dots, p-1$; their degrees have the form
\[
jd - \left(\left\lfloor \frac{ja_1}{p}\right\rfloor + \left\lfloor \frac{ja_2}{p}\right\rfloor\right)p.
\]
As $0\leq j<p$, we have
\[
0\leq\left\lfloor \frac{ja_1}{p}\right\rfloor + \left\lfloor \frac{ja_2}{p}\right\rfloor \leq (a_1-1)+(a_2-1) =d-2,
\]
so that every element of $D$ has the form $jd - \ell p$ with $0\leq \ell \leq d-2$. Mod $d$, this is $-\ell p$. By property~\ref{pro:one-in-each}, $d\neq p$, and also $0<d<2p$, so $p$ is relatively prime to $d$. Thus as $\ell$ ranges over all its possible values, namely $\{0,1,\dots,d-2\}$, $-\ell p$ ranges over all the residue classes mod $d$ except for $-(d-1)p = p\pmod{d}$. In particular, no element of $D$ can be congruent to $p$ mod $d$. This establishes property~\ref{pro:avoids-p-mod-d}.
\end{proof}

Using a small piece of Proposition~\ref{prop:hilbert-series}, we can show that in the situation that $G=\ZZ/p\ZZ$, $m=2$, the lower bound in  Theorem~\ref{thm:lower-bound} can be increased by $1$ plus a rounding error, and this is sharp:

\begin{proposition}\label{prop:Z-mod-p-lower}
If $G=\ZZ/p\ZZ$ with $p$ an odd prime, and $V$ is a finite-dimensional, faithful, non-modular representation of $G$, and the number of distinct nontrivial characters $m$ occurring in $V$ is $2$, then
\[
\bfield(G,V)=\gfield(G,V)\geq \lceil \sqrt{p}  + 1\rceil,
\]
with equality occurring if and only if either
\begin{itemize}
    \item $p=3$,  or
    \item $p=d^2-3d+1$ for a natural number $d\geq 4$, and \[
    [\Supp V]=[\{1,d^2-4d + 2\}].
    \]
\end{itemize}
\end{proposition}

\begin{remark} The first few primes $p>3$ for which the bound in Proposition~\ref{prop:Z-mod-p-lower} is attained are $5,11,19,29,41$. It is immediate from the proposition that the bound is sharp for infinitely many primes if and only if the quadratic polynomial $d^2-3d+1$ represents infinitely many primes. In particular, if Bunyakovsky's conjecture is true, then this bound is sharp for a sequence of arbitrarily large primes $p$.
\end{remark}

\begin{proof}[Proof of Proposition~\ref{prop:Z-mod-p-lower}]
Let
\[
d := \bfield(G,V)=\gfield(G,V).
\]
In the case $p=3$, by combining the Noether bound with Corollary~\ref{cor:gfield-at-least-3} we obtain $\bfield(G,V)=3=\lceil \sqrt{3} + 1\rceil$. So we assume $p\geq 5$ going forward.

Let $\bfa=(a_1,a_2),\bfb=(b_1,b_2)\in \NN^2$ be a lattice basis for $L(G,\Supp V)$ satisfying
\[
\max(\deg \bfa,\deg \bfb) = d.
\](This exists by the italicized statement in the proof of Proposition~\ref{prop:b=g-in-2d}.) Without loss of generality, suppose $\deg \bfa = a_1+a_2 = d$. Because $p\geq 5$ we have $\lceil\sqrt{p}+1\rceil < p$, so if $d=p$ there is nothing to prove. We assume going forward that $d<p$.

Then all of $a_i,b_i$ (for $i=1,2$) are less than $p$, so neither $\bfa$ nor $\bfb$ lies on a coordinate axis, and we conclude all $a_i$ and $b_i$ are positive, so $a_1 = d-a_2 \leq d-1$ and similarly $a_2\leq d-1$. By Proposition~\ref{prop:hilbert-series}, $\deg \bfb \neq \deg \bfa$, thus $\deg \bfb =b_1+b_2 \leq d-1$, and so  $b_1,b_2\leq d-2$.

By interchanging the axes if needed, we have
\[
a_1b_2 - a_2b_1 = p
\]
by Lemma~\ref{lem:index} (and the fact that the index of $L(G,\Supp V)$ in $\ZZ^2$ is the area of a fundamental parallelogram). If either $a_2$ or $b_1$ is $\geq 2$, or if $b_2\leq d-3$, then we claim $d > \sqrt p + 2$. If $b_2\leq d-3$, then
\[
\sqrt p < \sqrt{p+ a_2b_1}= \sqrt{a_1b_2} \leq \sqrt{(d-1)(d-3)} < d-2
\]
since $a_1\leq d-1$. If $b_1\geq 2$, then $b_2\leq (d-1)-2=d-3$ so this again applies. The case $a_2\geq 2$ is similar. Thus $d$ is strictly greater than $\lceil \sqrt p +1\rceil$ unless $a_2=b_1=1$ and $b_2=d-2$.

In this remaining case, we have $a_1=d-1$, so
\[
a_1b_2 - a_2b_1 = (d-1)(d-2)-1 = d^2 - 3d + 1 = p,
\]
and it is routine to verify that $d = \lceil \sqrt p + 1\rceil$. By substitution, the points $\bfa = (d-1,1)$ and $\bfb=(1,d-2)$ satisfy the equation
\[
a_1 + (d^2-4d + 2)a_2 = 0\pmod p,
\]
so, up to automorphisms of $G$, the set of characters $\{A_1,A_2\}$ is $\{1,d^2-4d+2\}$ as claimed.
\end{proof}

\section{Open questions}\label{sec:open-questions}

We hope that the present work stimulates further investigation of degree bounds for fields of rational invariants. In this section we present open questions. In Subsection~\ref{sec:conjecture}, we conjecture a sharp upper bound on $\bfield(\ZZ/p\ZZ,V)$ given the number $m$ of distinct nontrivial characters in $V$, and we discuss related results and questions. In Subsection~\ref{sec:other-questions}, we pose other questions raised by the present inquiry.

\subsection{A conjectural upper bound}\label{sec:conjecture}

It was mentioned above that the upper bound proven in Theorem~\ref{thm:upper-bound} is not sharp. In this section we conjecture a sharp upper bound, give some supporting evidence, exhibit representations that attain this conjectural bound, and pose related questions.

\begin{conjecture}\label{conj:sharp-upper}
If $G=\ZZ/p\ZZ$ with $p$ an odd prime, and $V$ is a representation of $G$ over a field of characteristic different from $p$, and $m$ is the number of distinct nontrivial characters occurring in $V$, then
\[
\bfield(G,V) \leq \left\lceil \frac{p}{\lceil m/2\rceil} \right\rceil.
\]
\end{conjecture}

Conjecture~\ref{conj:sharp-upper} emerged from computational data that is displayed in Table~\ref{tab:conjecture-data}. The computations were done in Magma. 

\begin{table}
    \centering
    \begin{tabular}{r|r r r r r r r r r r}
     & $m=1$ & 2 & 3 & 4 & 5 & 6 & 7 & 8 & 9 & 10\\
     \hline 
     $p=3$ & 3 & 3 & & & & & & & & \\
     5 & 5 & 5 & 3 & 3 & & & & & & \\
     7 & 7 & 7 & 4 & 4 & 3 & 3 & & & & \\
     11 & 11&11&6&6&4&4&3&3&3&3\\
     13 & 13&13&7&7&5&5&4&4&3&3\\
     17 & 17&17&9&9&6&6&5&5&4&4\\
     19 & 19&19&10&10&7&7&5&5&4&4\\
     23 & 23&23&12&12&8&8&6&6&5&5\\
     29 & 29&29&15&15& & & & & & \\
     31 & 31&31&16&16& & & & & & \\
     37 & 37&37&19&19& & & & & & 
    \end{tabular}
    \caption{Maximum value of $\bfield(\ZZ/p\ZZ,[S])$ over all equivalence classes $[S]$ of sets of $m$ distinct nontrivial characters. Computations done in Magma.}
    \label{tab:conjecture-data}
\end{table}

Another corroborating data point is that it is known \cite[Theorem~4.1]{bandeira2017estimation} that $\bfield(G,V_{\mathrm{reg}}) \leq 3$ if $G$ is any finite abelian group and $V_{\mathrm{reg}}$ is the regular representation over a field of characteristic coprime to $|G|$. In our framework, this is the situation that $m=|G|-1$, which forces $\Supp V$ to consist of every nontrivial character. Using methods of the present work, it is straightforward to check that this is an equality if $|G|\geq 3$.\footnote{In particular, if $|G|\geq 3$, then $L(G,\Supp V_{\mathrm{reg}})$ contains a degree $3$ point because $G$ has two (possibly equal) nontrivial characters whose product is also nontrivial. Therefore, $L(G,\Supp V_{\mathrm{reg}})$ is not contained in the sublattice $E$ of $\ZZ^{|G|-1}$ for which the sum of the coordinates is even; but it contains no degree $1$ point, so the sublattice of $L(G,\Supp V_{\mathrm{reg}})$ generated by the points of degree $\leq 2$ is contained in $E$.}  For the case $G = \ZZ/p\ZZ$ at hand, Conjecture~\ref{conj:sharp-upper} gives the right value in this situation. As an aside, it would have the added interesting implication that $\bfield(G,V)=3$ not only when $m=p-1$ but whenever $m\geq 2p/3$.

If Conjecture~\ref{conj:sharp-upper} is true, then the bound it gives is sharp. The following proposition exhibits, for given $p$ and $m$, a representation of $G=\ZZ/p\ZZ$ with $m$ distinct nontrivial characters that attains the upper bound in Conjecture~\ref{conj:sharp-upper}. In fact, the construction does not require that the order of $G$ be prime. 

\begin{proposition}\label{prop:extremal}
Let $G=\ZZ/n\ZZ$ with $n\geq 3$, and choose any $1 \leq m < n$. If $m$ is even, define
\[
S_m := \{ \pm 1, \pm 2, \dots, \pm m/2\}\subset\widehat G,
\]
where characters of $G$ are represented by integers as in equation~\eqref{eq:character-integer}. If $m$ is odd, define 
\[
S_m := \{ \pm 1, \pm 2, \dots, \pm (m-1)/2, (m+1)/2 \} \subset \widehat G.
\] 
In either case, we have
\[
\bfield(G,S_m) = \gfield(G,S_m) = \max\left(3,\left\lceil \frac{n}{\lceil m/2\rceil} \right\rceil\right).
\]
\end{proposition}

\begin{remark}
Note that in general,
\[
\left\lceil \frac{n}{\lceil m/2\rceil} \right\rceil \geq 3
\]
except in the special case that $n$ is even and $m=n-1$ exactly.

Also note that any representation $V$ of $G$ with $\Supp V = S_m$ is faithful, because the character $1\in S_m$ is already faithful. 
\end{remark}

\begin{proof}
It is convenient to work in $\ZZ^{S_m}$ and to index the coordinates $a_A$ by elements $A\in S_m$. Then the lattice $L(G,S_m)$ is defined by the equation
\[
\sum_{A\in S_m} Aa_A = 0 \pmod n.
\]
We first show that the sublattice $L_0(G,S_m)$ of rank $m-1$ defined by the equation
\begin{equation}\label{eq:key-hyperplane}
\sum_{A\in S_m} Aa_A = 0
\end{equation}
is generated in degree $\leq 3$, regardless of $n$ and $m$. We proceed by induction on $m$. The case $m=1$ holds vacuously because in this case \eqref{eq:key-hyperplane} has no nontrivial solutions, so $L_0(G,S_m)$ is trivial. The induction step splits into cases depending on the parity of $m$.

Even $m$: The sublattice $L_0(G,S_m)$ contains the degree-2 point $\bfb = (b_A)_{A\in S_m}$ defined by $b_{m/2}=b_{-m/2} = 1$ and the rest of the coordinates 0. The image of this point under the  projection $\pi_{-m/2}$ to the ${-m/2}$-coordinate is $1$, so generates $\ZZ$; in particular, it generates the full image of $L_0(G,S_m)$ under $\pi_{-m/2}$. The kernel of $\pi_{-m/2}$ is generated in degree $\leq 3$ because it is identified with $L_0(G,S_{m-1})$ (which is generated in degree $\leq 3$ by the induction hypothesis) by dropping the $a_{-m/2}=0$ coordinate, and this identification is degree-preserving. So $L_0(G,S_m)$ is generated in degree $\leq 3$ by Observation~\ref{obs:generate-over}.

Odd $m$: The argument is the same as the even case, except replacing the degree-2 point $\bfb$ above with a certain degree-3 point $\bfb'=(b'_A)_{A\in S_m}$ to be specified momentarily, and replacing $\pi_{-m/2}$ with the projection to the $(m+1)/2$-coordinate. For $m\geq 5$, $\bfb'$ is defined by  $b'_{(m+1)/2}=b'_{-(m-1)/2}=b'_{-1}=1$ and the rest of the coordinates 0. For $m=3$, the needed $\bfb'$ is defined by $b'_2=1$, $b'_{-1}=2$ (and $b'_1=0$).

Now we return to the full-rank lattice $L(G,S_m)$. Whether $m$ is even or odd, $S_m$ contains $\lceil m/2\rceil$. Define a point $\bfc=(c_A)_{A\in S_m}$ as follows. Divide $n$ by $\lceil m/2\rceil$ with remainder, yielding
\[
n = q \lceil m/2\rceil + r
\]
with $q$ an integer and $0 \leq r<\lceil m/2\rceil$. If $r>0$, then $r\in S_m$; set
\[
c_A :=
\begin{cases}
q, & A = \lceil m/2\rceil \\
1, & A = r \\
0, &\text{otherwise}.
\end{cases}
\]
If $r=0$, then set
\[
c_A :=
\begin{cases}
q, & A = \lceil m/2\rceil \\
0, &\text{otherwise}.
\end{cases}
\]
In all cases, all $c_A$ are nonnegative, and we have
\[
\sum_{A\in S_m}Ac_A = n,
\]
so in particular $\bfc\in L(G,S_m)$, and 
\[
\deg \bfc = \left\lceil \frac{n}{\lceil m/2\rceil} \right\rceil.
\]
Let $\varphi$ be the group homomorphism
\begin{align*}
\varphi : L(G,S_m) &\rightarrow \ZZ \\
\bfa &\mapsto \sum_{A\in S_m} Aa_A.
\end{align*}
Then $L_0(G,S_m)$ is the kernel of $\varphi$, and the image is $n\ZZ$, which is generated by $\varphi(\bfc)=n$. As we have shown above that $L_0(G,S_m)$ is generated in degree $\leq 3$, an application of Observation~\ref{obs:generate-over} yields that
\[
\bfield(G,S_m) \leq \max\left(3,\left\lceil \frac{n}{\lceil m/2\rceil} \right\rceil\right).
\]

Meanwhile, we claim that $L(G,S_m)$ does not contain any points of degree less than $\lceil n/\lceil m/2\rceil \rceil$ that are not already contained in $\ker \varphi = L_0(G,S_m)$, from which we can conclude that $\gfield(G,S_n)\geq \lceil n/\lceil m/2\rceil \rceil$ because $L_0(G,S_m)$ is not of full rank. We see the claim as follows. For any $\bfa$ in the nonnegative orthant we have
\[
|\varphi(\bfa)| \leq \sum_{A\in S_m} |A|a_A \leq \lceil m/2\rceil \deg \bfa,
\]
where the first inequality is the triangle inequality and the second is the fact that $|A| \leq \lceil m/2 \rceil$ for all $A\in S_m$.
In particular, if $\deg \bfa < \lceil n / \lceil m/2\rceil \rceil$, then the last number is less than $n$. If also $\bfa \in L(G,S_m)$, so that $\varphi(\bfa) \in n\ZZ$, it follows that $\varphi(\bfa)=0$. This proves the claim.

We now have
\[
\left\lceil \frac{n}{\lceil m/2\rceil} \right\rceil
\leq
\gfield(G,S_m)
\leq 
\bfield(G,S_m)
\leq
\max\left(3,\left\lceil \frac{n}{\lceil m/2\rceil} \right\rceil\right),
\]
so the proposition is established except when $\lceil n / \lceil m/2\rceil \rceil<3$, which only happens in the special case where $n$ is even and $m=n-1$, so that $\lceil n / \lceil m/2\rceil \rceil$ is 2. But because $n\geq 3$, Corollary~\ref{cor:gfield-at-least-3} shows that $\gfield(G,S_m)\geq 3$ even in this case, completing the proof.
\end{proof}

The computations that yielded Table~\ref{tab:conjecture-data} relied on the primality of $p$, so we do not have comparably systematic computational data for the case of $G=\ZZ/n\ZZ$ with $n$ composite. However, the description of the lattices in  Proposition~\ref{prop:extremal} makes sense for either prime or composite $n$, and they are extremal for $n=p$ prime in the range of values of $p$ and $m$ that we tested. It is natural to ask if they are always extremal:

\begin{question}\label{q:sharp-upper}
Does Conjecture~\ref{conj:sharp-upper} hold with not necessarily prime $n\geq 3$ in the place of $p$? (Assume $m<n-1$.)
\end{question}

The condition $m<n-1$ is added to Question~\ref{q:sharp-upper} to avoid having to replace $\lceil n / \lceil m/2\rceil\rceil$ with the more awkward $\max(3,\lceil n / \lceil m/2\rceil\rceil)$. The carved-out case is already handled: $m=n-1$ implies that $\Supp V$ contains every nontrivial character, so it is exactly $S_m$ with $m=n-1$, and we have $\bfield(G,V)=\gfield(G,V)=3$ by Proposition~\ref{prop:extremal}. This result can also be obtained by combining \cite[Theorem~4.1 and Section 4.3.3]{bandeira2017estimation} with Lemma~\ref{lem:delete-trivial}.

For $G=\ZZ/n\ZZ$ and odd $m$, the conjectured bound $\lceil n / \lceil m/2\rceil \rceil$ is attained by $\bfield(G,[S])$ not only with the $S=S_m$ defined in Proposition~\ref{prop:extremal}, but for $S$ obtained by dropping any one character from $S_{m+1}$. The proof has the same main ideas as Proposition~\ref{prop:extremal} but is more involved to write down. This yields up to $(m+1)/2$ distinct equivalence classes $[S]$ of $m$-subsets of $\widehat G\setminus\{0\}$, all of which attain the conjectured bound when $m$ is odd. (Equivalence is modulo automorphisms of $G$; the set of $m+1$ possible choices of a character to delete from $S_{m+1}$ is stable under the inversion automorphism.)

These classes for odd $m$, and the class $[S_m]$ for even $m$, are not the only extremal equivalence classes. For example, if the answer to Question~\ref{q:sharp-upper} is affirmative, then for any $n\geq 3$ and any $m\geq 2n/3$, the conjectural upper bound of Conjecture~\ref{conj:sharp-upper} matches the lower bound given in Corollary~\ref{cor:gfield-at-least-3}, thus {\em all} equivalence classes of $m$-sets are extremal. On the other hand, we observed that in the range of $n=p$ (prime) and $m$ for which we produced systematic computational data, these classes did tend to be the only extremal ones when we restricted attention to the highest values of $p$ we looked at for a fixed $m$. This prompts us to ask:

\begin{question}\label{q:unique-extremal}
For fixed even $m$ and all sufficiently high primes $p$, is $[S_m]$ (with $S_m$ defined by Proposition~\ref{prop:extremal}) the only equivalence class of $m$-subsets of $\widehat G\setminus\{0\}$ (up to automorphsims of $G$) that attains the bound in Conjecture~\ref{conj:sharp-upper}?

For fixed odd $m$ and all sufficiently high primes $p$, are the only equivalence classes attaining the bound in Conjecture~\ref{conj:sharp-upper} the classes $[S]$ of the sets $S$ obtained by dropping any one character from $S_{m+1}$?
\end{question}

\begin{question}
If we replace $p$ with not necessarily prime $n$ in Question~\ref{q:unique-extremal}, does the answer stay the same?
\end{question}

\subsection{Other questions}\label{sec:other-questions}

In this section, we articulate directions for further investigation, and also pose a series of more circumscribed questions arising from the above results (aside from those directly related to Conjecture~\ref{conj:sharp-upper}).

At the most general level, we can relax the restrictions on the groups considered. Although our methods largely committed us to working with abelian groups in coprime characteristic, and our study focused on cyclic groups (and particularly those of prime order), the questions make sense in greater generality. 

\begin{question}\label{q:upper-for-non-prime}
If $G$ is a finite cyclic group that is not of prime order, is there an analogue to Theorem~\ref{thm:upper-bound} establishing a gap between $\bfield(G,V)$ and $\bsep(G,V)$ for sufficiently high $m$ (assuming $\KK$ is algebraically closed or $\RR$)?
\end{question}

\begin{question}\label{q:upper-for-non-cyclic}
What if $G$ is abelian but not cyclic?
\end{question}

\begin{question}\label{q:upper-for-nonabelian}
What can be said if $G$ is not abelian?
\end{question}

It should be noted that, in the non-modular situation, cyclic groups are the only class of finite groups whose Noether number $\beta(G,V)$ ever attains the Noether bound. This was first shown by Schmid in characteristic zero  \cite[Theorem~1.7]{schmid1991finite}, and then by Sezer \cite{sezer2002sharpening} in general. (The same statement is immediate for $\bsep(G,V)$ as a consequence of the basic inequality $\bsep \leq \beta$.) In fact, by the results of \cite{cziszter-domokos} and \cite{cziszter2014noether}, for non-cyclic groups $G$, the Noether number $\beta(G,V)$ is already at most $|G|/2 + 2$. (See \cite{hegedHus2019finite} for more along this theme.) Thus, interesting bounds on $\bfield$ for non-cyclic $G$ would need to be lower than the bound given in Theorem~\ref{thm:upper-bound}.

We can also ask what happens if we loosen the restriction on the field characteristic, i.e., we allow $\operatorname{char} \KK \mid |G|$. In this (modular) situation, there is often a big gap between $\bsep$ and $\beta$: there is no universal bound on the Noether number of a representation, i.e., $\sup_V \beta(G,V) = \infty$ \cite{richman1996invariants}, while on the other hand, $\bsep(G,V)\leq |G|$ continues to hold \cite[Corollary~3.12.3]{derksen-kemper}. We also have $\bfield \leq |G|$ from \cite[Corollary~2.3]{fleischmann2007homomorphisms}. It would  be natural to investigate the possibility of a gap between $\bsep$ and $\bfield$ in this situation:

\begin{question}
Are there hypotheses on a modular representation $V$ of a finite group $G$ over an algebraically closed field $\KK$ that guarantee that $\bfield(G,V) < \bsep(G,V)$? That $\bfield(G,V) < c_1\bsep(G,V)+ c_2$ for some constants $c_1<1$ and $c_2$?
\end{question}

On the other hand, while it was mentioned in the introduction that $\bfield \leq \bsep$ for algebraically closed $\KK$ of characteristic zero, and the same inequality holds for abelian $G$ in arbitrary coprime characteristic by \cite[Theorem~2.1]{domokos}, it can fail if $\KK$ is sufficiently small (see note~\ref{note:sep-over-finite}). We  can ask if it ever fails when $\KK$ is algebraically closed:

\begin{question}
    Do there exist groups $G$ and  representations $V$ over an algebraically closed field $\KK$ of positive characteristic, for which $\bsep(G,V) < \bfield(G,V)$?
\end{question}

Another direction for further inquiry, suggested to us by Victor Reiner, is to seek  field analogues of the known degree bounds for modules of semi-invariants. For example, it is known  \cite[Corollary~1.3]{schmid1991finite} that in the characteristic zero case, the polynomial ring $\KK[V]$ is generated as a module over the subring of invariants $\KK[V]^G$ by elements of degree at most $|G|-1$. This gives an a priori bound on generators for the function field $\KK(V)$ as a vector space over the subfield $\KK(V)^G$ of rational invariants, but it seems likely that lower bounds can be given.

\begin{question}\label{q:vic's-q}
Given a finite-dimensional, non-modular representation $V$ of a finite group $G$, what can we say about the minimum natural number $d$ such that $\KK(V)$ is generated as a vector space over $\KK(V)^G$ by the elements of $\KK[V]_{\leq d}$?
\end{question}

Preliminary evidence suggests that, just like $\bfield$ as compared to $\beta$, the $d$ of Question~\ref{q:vic's-q} will in general be much lower than the comparable number for generating $\KK[V]$ as a module over $\KK[V]^G$. For example, for a faithful representation of $G=\ZZ/p\ZZ$, the bound $|G|-1$ is always sharp for $\KK[V]$ as a module over $\KK[V]^G$: if $x_1,\dots,x_N$ is a basis for $V^*$ on which $G$ acts diagonally, the module generated over $\KK[V]^G$ by $\KK[V]_{\leq d}$ cannot contain any $x_i^{p-1}$ unless $d\geq p-1$, since $\KK[V]^G$ contains no powers of $x_i$ below the $p$th. On the other hand, by adapting the methods of the present work, one can show for example that for $G=\ZZ/11\ZZ$ and any representation $V$ with $m=2$ distinct nontrivial characters, the field $\KK(V)$ is generated as a vector space over $\KK(V)^G$ by $\KK[V]_{\leq 5}$.

In addition to the above broad directions for further investigation, we also enumerate some more circumscribed questions that follow up the present inquiry.

The methods of proof of Propositions~\ref{prop:q+b+r-1} and \ref{prop:(p+3)/2} required the order of $G$ to be prime, but (extremely preliminary) computational evidence suggests that they hold without this restriction:

\begin{question}
Does Proposition~\ref{prop:q+b+r-1} hold with not-necessarily-prime $n$ in the place of $p$ (assuming $b$ as in the proposition exists)?
\end{question}

\begin{question}
Does Proposition~\ref{prop:(p+3)/2} hold with not-necessarily-prime $n$ in the place of $p$?
\end{question}

Likewise, the method of proof of Proposition~\ref{prop:gfield-at-least-3} required that $G$ be abelian, but a suitable modification of the statement is plausible without this restriction.

\begin{question}
Suppose $G$ is a not-necessarily-abelian finite group, and $V$ is a non-modular representation with at least one irreducible component that is not a one-dimensional representation coming from an involution in the character group of $G$'s abelianization. Do we have $\gfield(G,V) \geq 3$?
\end{question}

For $G=\ZZ/p\ZZ$, the lower bound of Theorem~\ref{thm:lower-bound} is improved by 1 (plus a rounding error) by Proposition~\ref{prop:Z-mod-p-lower} in the case that $V$ has $m=2$ distinct nontrivial characters. The proof method does not generalize beyond the $m=2$ case: it becomes possible for $m$ points of $d\Delta_m\cap \ZZ^m$, none of which lie on coordinate axes, to span a parallelotope of volume greater than $(d-1)^m$. Nonetheless, for $p$ and $m$ in the range for which we could generate systematic data, the natural generalization of the bound in Proposition~\ref{prop:Z-mod-p-lower} did seem to hold for arbitrarily many distinct nontrivial characters. So we ask:

\begin{question}
If $G=\ZZ/p\ZZ$ for $p$ an odd prime, and $V$ is a non-modular representation of $G$ with $m\geq 2$ distinct nontrivial characters, do we have $\gfield(G,V) \geq \lceil \sqrt[m]{p} + 1\rceil$?
\end{question}

We pose one final question, also suggested to us by Victor Reiner. In the special case where $G=\ZZ/p\ZZ$ and $m=\dim_\KK V=2$, Proposition~\ref{prop:hilbert-series} gives some information about the form of the Hilbert series of $\KK[V]^G$. The Hilbert series reflects structural information about a ring---for example, a celebrated result of Stanley \cite[Theorem~4.4]{stanley1978hilbert} states that a graded Cohen-Macaulay integral domain is Gorenstein if and only if its Hilbert series has a certain symmetry property. Therefore we ask:

\begin{question}\label{q:hilbert-shadow}
Suppose $G=\ZZ/p\ZZ$ and $V$ is a 2-dimensional representation consisting of two distinct nontrivial characters of $G$. Do the properties of the Hilbert series of $\KK[V]^G$ guaranteed by Proposition~\ref{prop:hilbert-series} reflect any interesting structural properties of the ring?
\end{question}

\subsection*{Acknowledgements}
The authors would like to thank Yeshua Flores for many fruitful discussions that advanced this project, William Unger for a key assist with setting up our computations in the Magma computer algebra system,  Victor Reiner for stimulating questions including Questions~\ref{q:vic's-q} and \ref{q:hilbert-shadow}, and for calling our attention to \cite{huffman}, and Reiner, Gregor Kemper, an anonymous referee, and especially Larry Guth for feedback on previous drafts. In addition, BBS wishes to thank Kemper, Guth, Reiner, Jonathan Niles-Weed, Alexander Wein, Swastik Kopparty, Lenny Fukshansky, Emilie Dufresne, M\'aty\'as Domokos, and Caleb Goldsmith-Spatz for helpful conversations. This work was supported by the Art of Problem Solving Initiative. All computations were done in Magma.

\bibliographystyle{alpha}
\bibliography{bib}

\end{document}